\documentclass[11pt,reqno]{amsart}
\usepackage{enumitem}
\usepackage[margin=2cm,bottom=3cm]{geometry}
\usepackage[dvipsnames]{xcolor}

\let\origleft\left
\let\origright\right
\renewcommand\left{\mathopen{}\mathclose\bgroup\origleft}
\renewcommand\right{\aftergroup\egroup\origright}

\usepackage{caption}
\usepackage{tikz}
\usetikzlibrary{calc,patterns}

\usepackage{thmtools}
\declaretheorem{theorem}
\declaretheorem[name=Corollary]{cortheorem}[sibling=theorem]
\declaretheorem[numberwithin=section]{lemma}
\declaretheorem[numbered=no]{remark}
\declaretheorem{proposition, corollary, scholium}[
  style=plain,
  sibling=lemma
]

\usepackage[colorlinks,allcolors=RoyalBlue]{hyperref}
\usepackage[capitalize,nameinlink]{cleveref}
\crefdefaultlabelformat{#2{\scshape #1}#3}
\creflabelformat{assump}{#2(#1)#3}
\crefname{assump}{Assumption}{Assumptions}
\crefname{figure}{Figure}{Figures}
\crefname{cortheorem}{Corollary}{Corollaries}

\usepackage{chngcntr}
\usepackage{apptools}
\AtAppendix{\counterwithout{proposition}{section}}

\usepackage{amssymb}
\usepackage{dsfont}
\usepackage{esint}
\usepackage{interval}
\usepackage{mathtools}
\usepackage{xparse}
\mathtoolsset{centercolon}

\newcommand\shortcut[3]{%
  \expandafter\xdef\csname #1#2\endcsname{\noexpand#3{#2}}%
}
\foreach \x in {A,...,Z} {%
  \shortcut{b}{\x}{\mathbf}
  \shortcut{c}{\x}{\mathcal}
  \shortcut{d}{\x}{\mathbb}
  \shortcut{f}{\x}{\mathfrak}
  \shortcut{r}{\x}{\mathrm}
}

\let\la\lambda%
\let\ka\kappa%
\let\ga\gamma%
\let\La\Lambda%
\let\eps\varepsilon%
\let\RR\dR%
\let\ZZ\dZ%
\intervalconfig{left open fence=(, right open fence=)}
\newcommand\newIntval[2]{\expandafter\DeclareDocumentCommand\csname #1%
        \endcsname{ o m m }{\IfValueTF{##1}{%
                        \interval[#2,##1]{##2}{##3}}{%
                        \interval[#2]{##2}{##3}}}}
\newIntval{oo}{open}
\newIntval{oc}{open left}
\newIntval{co}{open right}
\newIntval{cc}{}
\newcommand\bs[1]{\boldsymbol{\mathbf{#1}}}
\newcommand\dd{\mathrm d}

\DeclareMathOperator\ee{e}
\DeclareMathOperator\hull{hull}
\DeclareMathOperator\EE{\dE}
\DeclareMathOperator\PP{\dP}
\DeclareMathOperator\II{\mathds1}
\newcommand\vol[1]{\lvert#1\rvert}
\DeclareMathOperator\lb{\rB}
\DeclareMathOperator\cb{\rB}
\DeclareDocumentCommand\sf{ o g }{h^{\IfValueTF{#2}{(#2)}{(d)}}_{\IfValueTF{#1}{#1}{\la}}}
\newcommand\rvf[1][\la]{\rho^{(d)\vphantom{,1}}_{#1}}
\newcommand\Cover{\mathrm{Cover}}
\newcommand\Wd[1][1]{\operatorname{W_{#1}}}

\newcommand\mail[1]{\normalfont\href{mailto:#1}{\texttt{#1}}}

\author[P.~Calka]{Pierre Calka}
\address[PC]{Laboratoire de Mathématiques Raphaël Salem, Université de Rouen
Normandie, Avenue de l'Université, BP 12, Technopôle du Madrillet, F76801 Saint-Étienne-du-Rouvray France}
\email{\mail{pierre.calka@univ-rouen.fr}}%
\title[The support function of the high-dimensional Poisson polytope]{The support function of the high-dimensional\\Poisson polytope}

\author[B.~Dadoun]{Benjamin Dadoun}
\address[BD]{Laboratoire Manceau de Mathématiques, Le Mans Université, Avenue Olivier Messiaen, 72085 Le Mans France}%
\email{\mail{benjamin.dadoun@univ-lemans.fr}}%


\begin{document}
\begin{abstract}
    Let~$K^d_\la$ be the convex hull of the intersection of the homogeneous Poisson point process of intensity~$\la$ in~$\RR^d$, $d\ge 2$, with the Euclidean unit ball~$\dB^d$.
    In this paper, we study the asymptotic behavior as~$d\to\infty$
    of the support function
    \[\sf(u):=\sup_{x\in K^d_\la}\langle u,x\rangle\]
    in an arbitrary direction~$u\in\dS^{d-1}$
    of the Poisson polytope~$K^d_\la$.
    We identify three different regimes (subcritical, critical, and supercritical)
    in terms of the intensity~$\la:=\la(d)$ and derive in each regime
    the precise distributional convergence
    of~$\sf$ after suitable scaling. We especially treat this
    question where the support function is considered over multiple
    directions at once. We finally deduce partial counterparts for the radius-vector function of the polytope.

\end{abstract}
\maketitle

\thispagestyle{empty}
\section{Introduction}
The study of high-dimensional polytopes has attracted
recent attention in stochastic
geometry~\cite{Bonnet19,Gusakova21,Kabluchko23,Gusakova23}.
This is due to several factors. First of all, several classical models
of random polytopes can be defined in any dimension
and as such, they provide natural examples
which might confirm or deny several of the famous conjectures
from high-dimensional convex geometry.
For example, the work~\cite{Horrmann15} deals with Poisson
polytopes generated by random hyperplane tessellations
in the context of the hyperplane conjecture which claims that any $d$-dimensional convex body with volume~$1$ should have a section by a hyperplane whose $(d-1)$-dimensional volume is bounded from below by
a constant independent on the dimension~$d$. Similarly, the Hirsch conjecture which asserts that the edge-vertex graph of the boundary of a $d$-dimensional convex polytope with~$n$ facets should have a graph-diameter at most~$n-d$ was recently denied~\cite{Santos12} but since then, bounds in the form of polynomial functions of~$n$ and~$d$ have been investigated, in particular in the context of so-called random spherical polytopes~\cite{Bonnet22}. The second motivation behind the study of high-dimensional polytopes has an information theory background in the context of Shannon's channel coding theory and of the one-bit compression source coding. This involves in particular studying high-dimensional cells from the Poisson hyperplane tessellation for different regimes connecting the intensity with the dimension~\cite{Baccelli19,OReilly20}. Lastly, another context in which high-dimensional random polytopes play a central role is percolation in high dimension. Since Gordon's breakthrough result on the equivalent of the critical probability for the percolation on ${\mathbb Z}^d$ when~$d\to\infty$~\cite{Gordon91}, the high-dimensional setting is expected to possibly help estimating the phase transition. In particular, Poisson-Voronoi percolation along Voronoi cells generated by homogeneous Poisson points has attracted interest~\cite{Balister10,Conijn24}. 

In this paper, we consider the Poisson polytope~$K^d_\la$ obtained as the
convex hull of a Poisson point process with intensity~$\la:=\la(d)$
sampled in the Euclidean, $d$-dimensional unit ball~$\dB^d$.
Together with its binomial variant obtained when replacing
the Poisson point process with a fixed number~$n$
of i.i.d.\ variables uniformly distributed in~$\dB^d$,
this model of random polytope has proved to be one of the most
intensively studied in the literature.
In particular, several non-asymptotic results are known,
including Wendel's calculation~\cite{Wendel62}
of the probability that the origin belongs to the random
polytope, Efron and Buchta's mean-value
identities~\cite{Efron65,Buchta05} and the recent work due
to Kabluchko~\cite{Kabluchko21a,Kabluchko21b}
which provides an explicit formula for the expectation
of the $f$-vector constituted with the number of
$k$-dimensional facets of~$K^d_\la$.
The asymptotic study of~$K^d_\la$, when the dimension is
fixed and the intensity or the number of input points is large,
dates back to the seminal work due to R\'enyi and Sulanke
in dimension two~\cite{Renyi63,Renyi64} and has been then
carried through many subsequent works which have made
explicit limit expectations
for several functionals~\cite{Schneider80,Schutt94,Reitzner03},
variance bounds and limit variances~\cite{Barany10,Calka14}
and functional limit theorems for the support function and
radius-vector function~\cite{Calka13}.

In the continuation of the early work of B\'ar\'any and F\"uredi~\cite{Barany88},
who investigated the phase transition
for the probability of having all sampled points in convex position,
the high-dimensional study of~$K^d_\la$ and variants
has been continued over the last
decades~\cite{Dyer92,
Gatzouras09,
Bonnet19,
Chakraborti21,
Bonnet21},
with more emphasis on the threshold for the emergence of significant volumes.
Notably,
in~\cite{Bonnet21}, Bonnet, Kabluchko and Turchi estimated
the asymptotics for different intensity regimes of the mean volume
of a more general polytope called the $\beta$-polytope
which includes the case of~$K^d_\la$
(save for the fact
that they consider a binomial version of it)%
.
In particular, they exhibit a critical phase when the logarithm
of the number of input points is comparable to~$\frac{d}2\log d$,
i.e., they show that the expected normalized volume
of the polytope vanishes when the number of input points
increases slower than~$d^{d/2}$,
and persists when this number increases faster than~$d^{d/2}$.
At the critical threshold, when the number
grows like~$(d/2x)^{d/2}$, they prove a convergence to~$\ee^{-x}$.

Associated with the (random) polytope~$K^d_\la$ are the random processes~$\sf$ and~$\rvf$,
respectively called the \textit{support function} and the \textit{radius-vector function},
given
for any direction~$u$ in the unit sphere $\dS^{d-1}$ by
\begin{align}
    \sf(u)&:=\sup{\Bigl\{\langle u,x\rangle:x\in K^d_\la\Bigr\}},
    \label{eq:support-function}
    \intertext{and}
    \rvf(u)&:=\sup{\Bigl\{t>0:tu\in K^d_\la\Bigr\}}.
    \label{eq:radius-vector-function}
\end{align}
\begin{figure}[t]
    \centering
    \includegraphics[width=.33\textwidth]{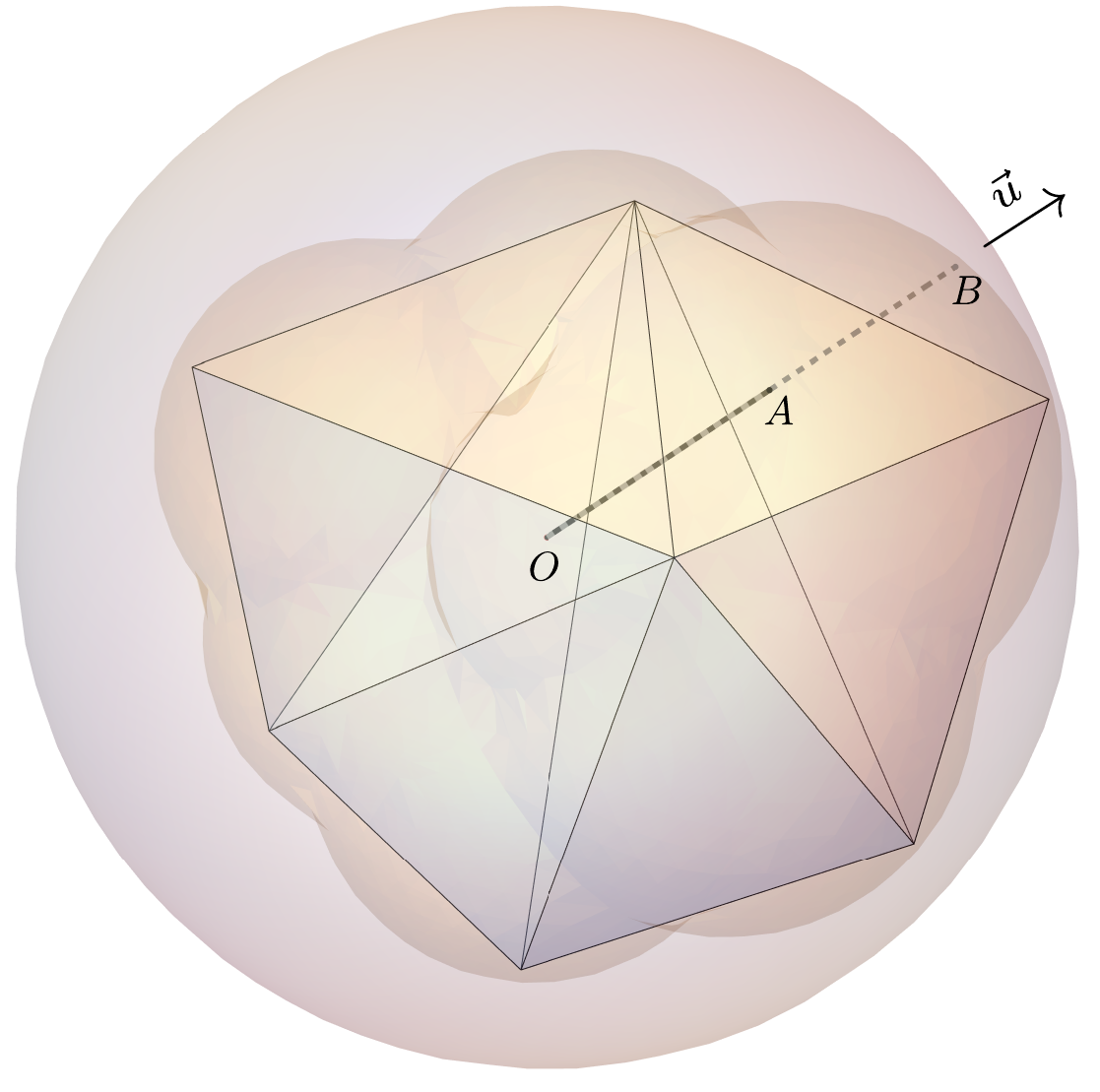}
    \caption{The support function $\sf(u):=OB$
    and the radius-vector function $\rvf(u):=OA$ of
    the polytope~$K^d_\la$ in dimension~$d=3$.
    The Voronoi flower is
    the union of the
    closed
    balls $(\frac x2+\frac{\lvert x\rvert}2
    \dB^d)$
    where~$x$ runs over the vertices of~$K^d_\la$.}%
    \label{fig:polytope}
\end{figure}%
The support function of~$K^d_\la$ corresponds
to the radius-vector function of its Voronoi flower%
, see \cref{fig:polytope}.
Even though~$\sf$ and~$\rvf$
are one-dimensional statistics,
they are known to fully characterize the convex body~\cite[Theorem~1.7.1]{Schneider13},
so understanding~$\sf$ or~$\rvf$ as~$d\to\infty$
already sheds significant light
on the asymptotic geometry of~$K^d_\la$.
To the best of our knowledge, the study of these functionals in high dimension has so far been largely ignored.
It appears however that the analysis of the radius-vector function
is more delicate and we mostly focus on the support
function in this work.

The distribution
of the random variables~$\sf(u)$
and~$\rvf(u)$ does not depend on~$u$
by rotational invariance; we let $\sf:=\sf(u)$ and $\rvf:=\rvf(u)$, and consider that~$u$
is either fixed in~$\dS^{d-1}$ or chosen uniformly at random in~$\dS^{d-1}$ (independently of the
Poisson process).
Given the following technical assumption on~$\la$,
\begin{equation}
    \liminf_{d\to\infty}\frac{\la\ka_d}d>2,\tag{H}
    \label[assump]{eq:assumption-origin}
\end{equation}
where~$\ka_d:=\lvert\dB^d\rvert$ is the $d$-dimensional volume of the unit ball,
the random variables~$\sf$ and~$\rvf$ will take
values in~$\cc{0}{1}$ with high probability as~$d\to\infty$
(see \cref{lem:origin}).
This hypothesis means that~$\la\ka_d$,
i.e., the mean number of points in~$\dB^d$ of the Poisson process,
is asymptotically greater than twice the dimension~$d$.

Our first result identifies the \emph{subcritical}, \emph{critical}, and
\emph{supercritical regimes}
where~$\sf$ tends either to~$0$, to a value in~$\oo{0}{1}$, or to~$1$ as~$d\to\infty$,
depending on whether $\log{\la\ka_d}$ grows slower than, comparably to,
or faster than the dimension~$d$.
\begin{theorem}[Asymptotic regimes of~$\sf$]\label{thm:regimes}
        Under \cref{eq:assumption-origin}, suppose that
    \[x:=\lim_{d\to\infty}\frac1d\log{\la\ka_d}\]
    exists in~$\cc{0}{\infty}$.
    Then the following holds in probability as~$d\to\infty$:
    \[\left\{\begin{aligned}
        \sf&\sim \sqrt{\frac2d\log{\la\ka_d}},&&\text{if~$x=0$},\\[.4em]
        \sf&\to\sqrt{1-\ee^{-2x}},&&\text{if $x\in\oo{0}{\infty}$},\\[.4em]
        1-\sf&\sim\frac12{(\la\ka_d)}^{-\frac2{d+1}},&&\text{if~$x=\infty$}.
    \end{aligned}\right.\]
\end{theorem}
\noindent%
Note that \cref{eq:assumption-origin}
is automatically fulfilled in the critical and supercritical regime ($x\in\oc{0}{\infty}$).

In addition, we determine the distributional fluctuations
of~$\sf$ around its limiting value.
\begin{theorem}[Convergence in distribution of the renormalized support function]\label{thm:main-one-directional}
  Under the assumption of \cref{thm:regimes},
  the random variable
  \[d\left(\log{\frac1{\sqrt{1-{\bigl(\sf\bigr)}^2}}}-\frac1{d+1}\log{\la\ka_d}\right)
  +\log{\sqrt{\mathfrak m(d)}}
  \]
  converges towards the standard Gumbel distribution
  as~$d\to\infty$, where
  \[\mathfrak m(d):=\begin{cases}
      4\pi\log{\la\ka_d},&\text{in the subcritical regime $\log{\la\ka_d}\ll d$},\\[.4em]
      2\pi d(1-\ee^{-2x}),&\text{in the critical regime $\log{\la\ka_d}\sim dx$ with $x\in\oo{0}{\infty}$},\\[.4em]
      2\pi d\left(1-{(\la\ka_d)}^{-\frac2{d+1}}\right),&\text{in the supercritical regime $\log{\la\ka_d}\gg d$}.
  \end{cases}
  \]
\end{theorem}
\begin{remark}
In the critical regime where $\log{\la\ka_d}=dx+y+o(1)$ with $x\in\oo{0}{\infty}$ and~$y\in\RR$, \cref{thm:main-one-directional} reads
  \begin{align*}
      d\left(\log{\frac1{\sqrt{1-{\bigl(\sf\bigr)}^2}}}-x\right)
      +\log{\sqrt{2\pi d(1-\ee^{-2x})}}&\:\xRightarrow[d\to\infty]{}\:G+y-x,
      \intertext{that is,}
      \sqrt{2\pi d}{\left(\frac{\ee^{-2x}}{1-{\bigl(\sf\bigr)}^2}\right)}^{\!-\frac d2}
      &\:\xRightarrow[d\to\infty]{}\:\frac{\ee^{G+y-x}}{\sqrt{1-\ee^{-2x}}},
  \end{align*}
  for some standard Gumbel variable~$G$.
\end{remark}

We now state analogous versions of the previous theorems when
the support function is considered over several directions at once.
Namely,
for a fixed integer~$m\ge2$, we define
\[\sf{d,m}\quad:=\inf_{u\in\dS^{d-1}\cap\RR^m}\sf(u).\]
Again, by rotational invariance,
the distribution of this infimum does not depend on the direction of
the linear $m$-dimensional section of~$\dS^{d-1}$.
In particular, we expect $\sf{d,m}$ to behave like~$\sf$, i.e.,
satisfy the exact same conclusions as \cref{thm:regimes}
with the same threshold.
This is confirmed by \cref{thm:main} below
which provides a Gumbel limit distribution
when~$\log \la\kappa_d$ belongs to the asymptotic range
$\bigl(\log d,d^{\frac32}\bigr)$.

\begin{theorem}[Distributional limit]\label{thm:main}
    Let~$m\ge2$ be a fixed integer, and let $\la:=\la(d)>0$
    satisfy
    either one of the three assumptions~\eqref{hyp:subcritical},~\eqref{hyp:critical}, or~\eqref{hyp:supercritical}
    given in \cref{lem:condition_janson}.
    Then there exist two explicit sequences~$\mathfrak a(d;m)$
    and~$\mathfrak b(d;m)$ (given at~\eqref{eq:a-multidirectional} and~\eqref{eq:b-multidirectional}) such that
    \[
        \mathfrak a(d;m)
        -\mathfrak b(d;m)\log{\frac1{\sqrt{1-{\bigl(\sf{d,m}\bigr)}^2}}}
    \]
    converges in law as~$d\to\infty$ towards the standard Gumbel
    distribution.
\end{theorem}
The proof of \cref{thm:main}
relies essentially on the \textit{ad hoc} application of a remarkable result due to S.~Janson on random
coverings of a set~\cite{Janson86}.
As a corollary, we extend \cref{thm:regimes} to~$\sf{d,m}$.
\begin{cortheorem}[Asymptotic regimes of~$\sf{d,m}$]\label{cor:regimes-multidirectional}
  Under the assumptions of \cref{thm:main}, let
\[x:=\lim_{d\to\infty}\frac1d\log{\la\ka_d}\in\cc{0}{\infty}.\]
Then the following holds in probability as~$d\to\infty$:
\[\left\{\begin{aligned}
  \sf{d,m}&\sim \sqrt{\frac2d\log{\la\ka_d}},&&\text{if~$x=0$},\\[.4em]
  \sf{d,m}&\to\sqrt{1-\ee^{-2x}},&&\text{if $x\in\oo{0}{\infty}$},\\[.4em]
  1-\sf{d,m}&\sim\frac12{(\la\ka_d)}^{-\frac2{d+1}},&&\text{if~$x=\infty$}.
\end{aligned}\right.\]
\end{cortheorem}
Though the idea is quite natural, the study of the asymptotics of the support function as a tool for describing the geometry of high-dimensional random polytopes seems original. We provide here an extensive analysis of its minimum on a subspace of fixed dimension with the help of essentially two ingredients: on one hand, a careful application of known estimates of the incomplete beta function and on the other hand, a reformulation of the problem in terms of the probability to cover the sphere with random spherical caps. The latter requires in particular a slight extension of a precise estimate due to Janson~\cite{Janson86} of the covering probability with i.i.d.\ rescaled random spherical caps when the common distribution of the geodesic radii of the caps is allowed to depend on the scaling parameter. This result is interesting on its own and is deduced by a coupling technique which avoids rewriting the details of Janson's original proof in our context. Besides, our work is a possible first step towards a more systematic study of the whole support function process along the $d$-dimensional sphere, see the discussion below. The radius-vector function process is certainly worthy of investigation as well and we state a minor first result in this direction deduced from Wendel's probability calculation~\cite{Wendel62}.

\cref{thm:regimes,cor:regimes-multidirectional} are clearly reminiscent of~\cite[Theorem~3.1]{Bonnet21}.
If~$\la$ is taken so that
$\log{\la\ka_d}$
belongs to the asymptotic range~$[d,d\log d)$, then the expected volume ratio
$\EE{\lvert{K^d_\la}\rvert}/\ka_d$ vanishes as observed in~\cite{Bonnet21},
while according to \cref{thm:regimes,cor:regimes-multidirectional}, $K^d_\la$ still has long `arms'
in any finite number of independent directions.
This confirms the well-known picture of a high-dimensional
convex body which was popularized by Vitali Milman
and which looks like a `star-shaped body with a lot of
points very far from the origin and lot of points very
close to the origin'~\cite[Section~2]{Guedon14}.
Studying the minimum of the support function over a section
of~$K^d_\la$ may provide a way of quantifying the size of
the `holes', i.e., estimating the critical dimension under
which the support function is close to one in every direction
of the section of~$K^d_\la$ and above which we expect to see
directions almost unoccupied by the section of~$K^d_\la$.
Unsurprisingly, \cref{cor:regimes-multidirectional} suggests that as
soon as we reach the threshold for the one-dimensional
section given in \cref{thm:regimes},
we expect every section with fixed dimension~$m$ to
look like the~$m$-dimensional unit ball.
The rest of the study should then lead us to consider~$m$
tend to infinity with~$d$ in the supercritical
case of \cref{thm:regimes} in order to decide
when exactly the function~$\sf{d,m}$ switches
from being almost equal to~$1$ to being almost equal to~$0$.
This requires a serious revision of the covering methods
used in the proof of \cref{thm:main} that we leave
for further work.

We expect that our results can be depoissonized and that, as long as the asymptotics of the measure of a spherical cap is accessible, our methods should extend to other models with rotationally-invariant points, such as the $\beta$-polytope studied in~\cite{Bonnet19} which covers our model when~$\beta=0$ and the case of the convex hull of points uniformly distributed on the unit sphere when~$\beta=-1$. In the latter setting, Bonnet and O'Reilly~\cite{Bonnet22b} have recently studied the typical height of a facet, which, incidentally, is proved to have the same first-order asymptotics as the support function in a fixed direction, as can be seen from a comparison between Theorem~4 therein and our \cref{thm:regimes}.

The paper is structured as follows. We start in \cref{sec:preliminaries}
with some asymptotic preliminaries, notably for the
incomplete beta function to which the tail probabilities of~$\sf$ and~$\sf{d,m}$ are
related. As a warm-up, in \cref{sec:one-directional-support-function}, we establish
\cref{thm:regimes,thm:main-one-directional} by elementary means.
\cref{sec:multidirectional-support-function} is devoted to the use of Janson's covering
techniques for proving \cref{thm:main,cor:regimes-multidirectional}.
Finally, in \cref{sec:radius-vector-function}, we transfer some of our results
to the radius-vector function~$\rvf$ of~$K^d_\la$.

\paragraph*{\bfseries Notation.}
    Unless otherwise specified, all asymptotic estimates are w.r.t.\ $d\to\infty$.
    The relation~$g\gg f$ (or~$f\ll g$, or~$f=o(g)$)
    for nonnegative~$f:=f(d)$ and~$g:=g(d)$
    means that $f(d)\le\eps g(d)$
    holds for all~$d$ sufficiently large and any~$\eps>0$, while~$f=O(g)$
    or~$f\lesssim g$
    indicate that $f(d)\le Cg(d)$ holds for all~$d\ge1$ and some constant~$C>0$.
    We also write~$f\sim g$ if $\lvert f-g\rvert\ll g$,
    and~$f\asymp g$ if both~$f\lesssim g$ and~$g\lesssim f$ hold.

\section{Preliminaries}\label{sec:preliminaries}
This section aims at providing an explicit formula for the distribution function
of%
~$\sf$
in terms of the incomplete beta function.
It also paves the way for the
asymptotic study
of the tail probability of~$\sf$
and~$\sf{d,m}$,
which is the focus of \cref{sec:one-directional-support-function,sec:multidirectional-support-function}.

Let $\cP^d_\la,\,d\ge2,$ be Poisson point processes
(embedded in a common abstract probability space~$(\Omega,\cA,\PP)$)
with intensities $\la:=\la(d)>0$ in~$\RR^d$.
The polytope~$K^d_\la$ is defined as the convex hull
of~$\cP^d_\la\cap\dB^d$,
where $\dB^d:=\{x\in\RR^d:\lvert x\rvert\le1\}$ is the Euclidean unit ball
of~$(\RR^d,\|{{}\cdot{}}\|)$, with the Euclidean norm~$\|{{}\cdot{}}\|$
derived from the usual inner product~$\langle\cdot,\cdot\rangle$. We recall that $\lvert{{}\cdot{}}\rvert$ denotes the $d$-dimensional Lebesgue measure of~$\RR^d$. In both notations~$\|{}\cdot{}\|$ and $\lvert{{}\cdot{}}\rvert$, the dependency on~$d$ is implicit.

A~priori, the support function~$\sf$ introduced in~\eqref{eq:support-function} takes
values in~$\{-\infty\}\cup\oo{-1}{1}$, but if~$\la$ is not chosen too small,
then the polytope~$K^d_\la$ will likely contain the origin, which means that~$\sf\ge0$.
\begin{lemma}[Containing the origin]\label{lem:origin}
    If \cref{eq:assumption-origin} holds,
    then $\PP(0\in K^d_\la)\to1$ as~$d\to\infty$.

    \noindent
    If however $\limsup_{d\to\infty}\frac{\la\ka_d}d<2$,
    then $\PP(0\in K^d_\la)\to0$ as~$d\to\infty$.
\end{lemma}
\begin{proof}[Proof of \cref*{lem:origin}]
    The number $N:=N(d)$ of points in~$K^d_\la$
    has a Poisson distribution with
    mean~$\la\ka_d$.
    Further, conditional on~$N$, those points have a symmetric distribution in~$\RR^d$.
    It follows from Wendel's formula~\cite{Wendel62} that
    \begin{align*}
        \PP\bigl(0\notin K^d_\la\;\big|\;N\bigr)
        &=\II_{\{N\le d\}}
        +\II_{\{N>d\}}2^{-(N-1)}\sum_{k=0}^{d-1}\binom{N-1}k\\[.4em]
        &=1-\II_{\{N>d\}}2^{-(N-1)}\sum_{k=d}^{N-1}\binom{N-1}k\\[.4em]
        &=1-\PP(S_{N-1}\ge d\mid N),
        \intertext{where $S_n,\,n\ge0,$ are Binomial$(n,\frac12)$ random
                variables independent of~$N$. Hence}
    \PP\bigl(0\in K^d_\la)&=\PP(S_{N-1}\ge d).
  \end{align*}
  By the law of large numbers,
  $\PP(S_{N-1}\ge d)\to1$ as~$d\to\infty$ if
  \[\liminf_{d\to\infty}\frac1d\EE[S_{N-1}]>1,\]
  that is (since $\EE[S_{N-1}]=\frac{\la\ka_d-1}2$), if
  $\la\ka_d\ge(2+\eps)d$
  holds for some~$\eps>0$ and all~$d$ sufficiently large;
  similarly, $\PP(S_{N-1}\ge d)\to0$ if
  $\la\ka_d\le(2-\eps)d$ holds for some~$\eps>0$ and all~$d$ sufficiently large. (In fact, by classical large deviation
  theory, these two convergences occur exponentially fast.)
\end{proof}

Taking \cref{eq:assumption-origin} for granted,
we thus have $0\le \sf\le1$ w.h.p.\ as~$d\to\infty$. Now,
\[\PP\left(\sf\le r\right)=\PP\left(\sf(u)\le r\right),\quad0\le r\le 1,\]
for, e.g., $u:=(1,0,\ldots)\in\dS^{d-1}$; we compute this probability as
\[\PP\left(\cP^d_\la\cap\cC^d(r;u)=\emptyset\right)
=\ee^{-\la\vol{\cC^d(r;u)}},\]
where the spherical cap $\cC^{d}(r;u):=\{x\in\dB^d:\langle u,x\rangle>r\}$ has volume
\begin{equation}
\vol{\cC^{d}(r;u)}
=\ka_{d-1}\int_r^1{(1-t^2)}^{\frac{d-1}2}\,\dd t
=\frac{\ka_{d-1}}2\int_0^{1-r^2}v^{\frac{d-1}2}{(1-v)}^{-\frac12}\,\dd v,\label{eq:volume-spherical-cap}
\end{equation}
with the last integral resulting from the change of variable $v\gets1-t^2$.
Hence
\begin{equation}
    \PP\left(\sf\le r\right)=\exp\left(-\frac{\la\ka_{d-1}}2\lb\bigl(1-r^2;\tfrac{d+1}2,\tfrac12\bigr)
    \right),\label{eq:cdf-support-function}
\end{equation}
where
\begin{equation*}
    \lb(x;p,q):=\int_0^xv^{p-1}(1-v)^{q-1}\,\dd v%
    ,\quad x\in\cc{0}{1},\,p,q>0,%
    \end{equation*}
is the \emph{lower incomplete beta function} (the complete beta function
is $\cb(p,q):=\lb(1;p,q)$).
We will see in \cref{sec:multidirectional-support-function}
that when considering the support function
over~$m\ge2$
directions at once, the distribution function of the
infimum~$\sf{d,m}$ also involves this special function
(with the third parameter~$q=\frac12$
replaced by~$q=\frac m2$)
as well as the volume of unit balls.
Thus, the asymptotic behavior of~$\sf$ and of~$\sf{d,m}$
will depend on the interplay between
the intensity~$\la:=\la(d)$
and the two quantities~$\ka_n$ and~$\lb(x;p,q)$,
where~$n,x$ and~$p$ may depend on the dimension~$d$.

Regarding~$\ka_d$,
we will essentially use the asymptotic relation
\begin{equation}
    \ka_{d-1}
   =\ka_d\,\sqrt\frac d{2\pi}\Biggl[1+O\left(\frac1d\right)\Biggr],\label{eq:rapport-kappa}
\end{equation}
which is easily derived from the well-known formula
\begin{equation}
  \ka_d=\frac{\pi^{\frac d2}}{\Gamma\left(1+\frac d2\right)}
  \label{eq:volume-ball}
\end{equation}
and Stirling's formula for Euler's gamma function~$\Gamma$
(see, e.g.,~\cite[(3.24)]{Temme96}).
As for the incomplete beta function, we can derive basic first-order estimates:
\begin{lemma}[Estimates of the incomplete beta function]\label{lem:estimates-incomplete-beta-function}
        Let~$p,q>0$.
    Then for any $x\in\oo{0}{1}$ such that $(p+q)x<p+1$,
    \[\text{the ratio}\quad{\lb(x;p,q)}\:\Big/\:{\frac{x^p{(1-x)}^{q-1}}p}\quad
    \text{lies between~$1$ and $\dfrac1{1-\frac{q-1}{p+1}\cdot\frac x{1-x}}$}.\]
    In particular, when $x\in\oo{0}{1}$ and~$p,q>0$ are three sequences indexed by~$d$:\\[.4em]

    \noindent 1.\enspace If $p\gg\dfrac{|q-1|x}{1-x}$, then
        \begin{align*}
                \lb(x;p,q)&=\frac{x^p{(1-x)}^{q-1}}p\left[1+O\left(\frac{|q-1|x}{p(1-x)}\right)\right].
        \intertext{2.\enspace If $p\gg\dfrac{|q-1|}{1-x}$, then}
        \frac{\lb(x;p,q)}{\cb(p,q)}&=\frac{x^p{\left[(1-x)p\right]}^{q-1}}{\Gamma(q)}\left[1+O\left(\frac{|q-1|}{p(1-x)}\right)\right].
        \end{align*}
\end{lemma}
\begin{proof}
    For $x\in\oo{0}{1}$
    we have from~\cite[(11.33)]{Temme96} the series representation
    \[B(x;p,q)=\frac{x^p{(1-x)}^q}p\sum_{n=0}^\infty\frac{{(p+q)}_n}{{(p+1)}_n}\,x^n,\]
    where the ratio of Pochhammer symbols
    \[
        \frac{{(p+q)}_n}{{(p+1)}_n}:=\frac{p+q}{p+1}\cdot\frac{p+q+1}{p+2}\cdots\frac{p+q+n-1}{p+n}
    \]
    belongs to $\cc[scaled]{1}{{(\frac{p+q}{p+1})}^{\!n}}$ if~$q\ge1$,
    and to $\cc[scaled]{{(\frac{p+q}{p+1})}^{\!n}}{1}$ if~$0<q<1$.
    Now if $(p+q)x<p+1$, then
    \[\sum_{n=0}^\infty{\left(\frac{p+q}{p+1}\right)}^{\!n}x^n
    =\dfrac1{1-\frac{p+q}{p+1}x}=\dfrac1{1-x}\cdot\frac1{1-\frac{q-1}{p+1}\cdot\frac x{1-x}},\]
    hence the lower and upper bounds on $\lb(x;p,q)/{\frac{x^p{(1-x)}^{q-1}}p}$.
    The first stated asymptotic estimate is an immediate consequence of these bounds,
    while the second one follows from the first one combined with
    \[\cb(p,q)=\Gamma(q)\cdot\frac{\Gamma(p)}{\Gamma(p+q)}=\frac{\Gamma(q)}{p^q}\left[1+O\left(\frac{|q-1|}p\right)\right],\]
    see, e.g.,~\cite[(3.31)]{Temme96}.
\end{proof}

\section{The support function in one direction}\label{sec:one-directional-support-function}
We are now ready to establish \cref{thm:regimes,thm:main-one-directional}:
we do so by recalling the distribution function~\eqref{eq:cdf-support-function}
of the support function~$\sf$, then plug in
asymptotics for the incomplete beta function (\cref{lem:estimates-incomplete-beta-function})
and for the volume of the Euclidean unit balls~\eqref{eq:rapport-kappa}.

\noindent To start with, we identify the limit of~$\sf$ in probability.
\begin{lemma}[Limit in probability of the support function]\label{lem:limit-support-function}
  Under the assumption of \cref{thm:regimes},
  \[\lim_{d\to\infty}\sf=\sqrt{1-\ee^{-2x}},\quad\text{in probability}\]
  (with the convention $\ee^{-\infty}=0$).
\end{lemma}
\begin{proof}
    We start by applying \cref{lem:estimates-incomplete-beta-function} with only the second argument of
    the incomplete beta function depending on~$d$: for any fixed $r\in\oo{0}{1}$,
    \begin{align}
        \lb\left(1-r^2;\tfrac{d+1}2,\tfrac12\right)
        &=\frac{2{(1-r^2)}^{\frac{d+1}2}}{rd}\left[1+O\left(\frac1d\right)\right].\label{eq:asymptotics-incomplete-beta}
        \intertext{Inserting this and
        the other estimate~\eqref{eq:rapport-kappa} into~\eqref{eq:cdf-support-function} then yields}
        -{\log{\PP\left(\sf\le r\right)}}
        &=\frac{\la\ka_d{(1-r^2)}^{\frac{d+1}2}}{r\sqrt{2\pi d}}\left[1+O\left(\frac1d\right)\right]\notag\\[.4em]
        &=\exp\left\{\log{\la\ka_d}+\frac{d+1}2\log(1-r^2)
        -\log{r\sqrt{2\pi d}}+O\left(\frac1d\right)\right\}\label{eq:exponent-cdf-support-function}\\[.4em]
        &=\exp\left\{d\left(\frac{\log{\la\ka_d}}d+\log{\sqrt{1-r^2}}+o(1)\right)\right\}.\notag
    \intertext{%
    Since $\frac{\log{\la\ka_d}}d\to x\in\cc{0}{\infty}$ as~$d\to\infty$, the change of sign
    in this exponent provides the required threshold,
    i.e.,}
    \lim_{d\to\infty}\PP\left(\sf\le r\right)
        &=\begin{cases}
            0&\text{if $r<\sqrt{1-\ee^{-2x}}$},\\
            1&\text{if $r>\sqrt{1-\ee^{-2x}}$}.
        \end{cases}\notag
    \end{align}
    This proves the convergence in distribution of~$\sf$ towards the constant $\sqrt{1-\ee^{-2x}}$, which is equivalent to the convergence in probability to the same limit.
\end{proof}
By also letting the first argument of the incomplete beta function depend on~$d$,
a deeper application of \cref{lem:estimates-incomplete-beta-function} enables us to complete the proof of
Theorems~\ref{thm:regimes} and~\ref{thm:main-one-directional}.
\begin{proof}[Proof of \cref*{thm:main-one-directional}]
    We apply \cref{lem:estimates-incomplete-beta-function} again
    by letting $r\in\oo{0}{1}$ in the previous proof depend on~$d$.
    Whenever~$dr^2\gg1$ holds, we may still write
    (instead of~\eqref{eq:asymptotics-incomplete-beta})
    \[
      \lb\left(1-r^2;\tfrac{d+1}2,\tfrac12\right)
        \sim\frac{2{(1-r^2)}^{\frac{d+1}2}}{rd},
    \]
    which leads to the following weakening of~\eqref{eq:exponent-cdf-support-function}:
\begin{equation}
  {-{\log{\PP\left(\sf\le r\right)}}}
    =\exp\left\{\log{\la\ka_d}+\frac{d+1}2\log(1-r^2)
    -\log{r\sqrt{2\pi d}}+o(1)\right\}.\label{eq:exponent-cdf-support-function-2}
\end{equation}
    Fix~$\tau\in\RR$ and choose $r:=r(d;\tau)$ as the unique solution 
    to the equation
    \begin{equation}
    \log{\la\ka_d}+\frac{d+1}2\log(1-r^2)
    -\log{r\sqrt{2\pi d}}
    =-\tau,\quad\text{namely}
    \enspace
    \frac{{(1-r^2)}^{\frac{d+1}2}}r=\frac{\sqrt{2\pi d}}{\la\ka_d}\ee^{-\tau}.%
    \label{eq:implicit-definition-r}
    \end{equation}
    In particular, under~\cref{eq:assumption-origin},
    \[-\frac{d+1}2\log(1-r^2)\ge\log{\frac{\la\ka_d}{\sqrt{2\pi d}}}+\tau\to\infty,\]
    so that, indeed, $dr^2\gg1$ and~\eqref{eq:exponent-cdf-support-function-2} is true.
    Inserting~\eqref{eq:implicit-definition-r}
    there then yields
    \[\lim_{d\to\infty}\PP\left(\sf\le r\right)
    =\ee^{-\ee^{-\tau}},\]
    which is the c.d.f.\ of the standard Gumbel distribution.
    Now, on the one hand, we observe that
    \begin{align*}
        \PP\left(\sf\le r\right)
        &=\PP\left[-\frac{d+1}2\log\left(1-{\bigl(\sf\bigr)}^2\right)-\log\frac{\la\ka_d}{r\sqrt{2\pi d}}\le\tau\right]\\[.4em]
        &=\PP\left[d\left(\log{\frac1{\sqrt{1-{\bigl(\sf\bigr)}^2}}}-\frac1{d+1}\log{\la\ka_d}\right)
        +\frac{\log{r\sqrt{2\pi d}}}{1+O\left(\frac1d\right)}\le\tau+O\left(\frac1d\right)\right].
    \end{align*}
    On the other hand, we have from~\eqref{eq:implicit-definition-r},
    \begin{align*}
      1-r^2&=r^{\frac2{d+1}}{(\la\ka_d)}^{-\frac2{d+1}}\left[1+O\biggl(\frac{\log d}d\biggr)\right],
    \intertext{and}
      \frac1d\log r&=\frac1d\log{\la\ka_d}
      +O\left(-{\log\bigl(1-r^2\bigr)}\right),
    \end{align*}
    from which it follows that
    \[\left\{\begin{aligned}
        r^2&\sim\dfrac2d\log{\la\ka_d},&&\text{if $\log{\la\ka_d}\ll d$},\\[.4em]
        r^2&\to1-\ee^{-2x},&&\text{if $\log{\la\ka_d}\sim dx$ with $x\in\oo{0}{\infty}$},\\[.4em]
        1-r^2&\sim{(\la\ka_d)}^{-\frac2{d+1}},&&\text{if $\log{\la\ka_d}\gg d$}.
    \end{aligned}\right.\]
    This allows us to conclude that
    \[\PP\left[d\left(\log{\frac1{\sqrt{1-{\bigl(\sf\bigr)}^2}}}-\frac1{d+1}\log{\la\ka_d}
    \right)+\log{\sqrt{\mathfrak m(d)}}\le\tau\right]
    \xrightarrow[d\to\infty]{}\:\ee^{-\ee^{-\tau}},\]
    with~$\mathfrak m(d)$ as stated.
\end{proof}

\begin{proof}[Proof of \cref*{thm:regimes}]
  It follows from \cref{thm:main-one-directional} that
  \[
      d\left(\log{\frac1{\sqrt{1-{\bigl(\sf\bigr)}^2}}}-\frac1{d+1}\log{\la\ka_d}
      \right)+\log{\sqrt{\mathfrak m(d)}}
      =O_\dP(1),
  \]
  where $X(d)=O_{\dP}(1)$ means that
  $\lim_{A\to\infty}\limsup_{d\to\infty}\PP(|X(d)|>A)=0$.
  Multiplying this equation by~$-\frac2d$ and taking the exponential function,
  we deduce that
  \[{(\la\ka_d)}^{\frac2{d+1}}\left(1-{\bigl(\sf\bigr)}^2\right)
  =1+\frac1{\log d}\cdot O_\dP\left(1\right),\]
  where we discarded the ${\mathfrak m(d)}^{-\frac1d}$ term
  because $\mathfrak m(d)=O(d)$.
  In the subcritical regime
  $\log{\la\ka_d}\ll d$, we obtain that
  in probability as~$d\to\infty$,
  \begin{align*}
  \sf&\sim\sqrt{\frac2d\log{\la\ka_d}}
\intertext{(because then
${1-(\la\ka_d)}^{\frac2{d+1}}\sim-{\frac2d\log{\la\ka d}}$).
In the supercritical regime $\log{\la\ka_d}\gg d$, we get instead}
  1-\sf&\sim\frac12{(\la\ka_d)}^{-\frac2{d+1}}
  \end{align*}
  because, by \cref{lem:limit-support-function}, $1+\sf\to2$ in probability.
  This completes the proof of \cref{thm:regimes}.
\end{proof}

\section{The infimum of the support function over multiple directions}\label{sec:multidirectional-support-function}
In this section,
we extend the study of the asymptotic behavior of the support function
when considered over several directions simultaneously. Namely,
we consider
\[\sf{d,m}\;:=\inf_{u\in\dS^{d-1}\cap\RR^m}\sf(u),\]
where~$m\ge2$ is a fixed integer. \Cref{sec:reduction-to-covering-problem} consists in reinterpreting the distribution of~$\sf{d,m}$ in terms of a covering probability. We then present in~\cref{sec:janson-covering} a remarkable covering technique due to Janson~\cite{Janson86} which we specialize to our setting. This finally allows us to prove \cref{thm:main}.
\subsection{Reduction to a covering problem.}\label{sec:reduction-to-covering-problem}
We start by relating the tail event~$\sf{d,m}\ge r$
to the event of covering the sphere~$\dS^{m-1}:=\{y\in\RR^m:\|y\|=1\}$
with i.i.d.\ geodesic balls. In this direction,
we write $v_{\dS^{m-1}}(\dd x)$
for
the $(m-1)$-dimensional surface measure on~$\dS^{m-1}$
(so that $v_{\dS^{m-1}}(\dS^{m-1})=m\ka_m$),
and
\[B_{\dS^{m-1}}(x,\theta):=\Bigl\{y\in\RR^m:\|y\|=1\text{ and }\langle x,y\rangle>\cos\theta\Bigr\}\]
for
the geodesic ball in~$\dS^{m-1}$ with center~$x\in\dS^{m-1}$ and radius~$\theta\in\oc{0}{\pi}$.
\begin{lemma}[Covering the sphere]\label{lem:reduction-to-covering-problem}
  For every $r\in\oo{0}{1}$, we have
  \begin{equation}
    \sf{d,m}\ge r\iff
    \dS^{m-1}%
    =
    \bigcup_iB_{\dS^{m-1}}(x_i,a\rho_i),\label{eq:covering-sphere}
  \end{equation}
  where
  \begin{equation}
    a:=\frac1{\sqrt d}\,\frac{\sqrt{1-r^2}}r,\label{eq:vanishing-parameter-a}
  \end{equation}
  and the centers and radii~$(x_i,\rho_i)$ arise
  as the atoms of a Poisson point process on~$\dS^{m-1}\times\oo{0}{\infty}$
  whose intensity measure $\La^{(d,m)}_rv_{\dS^{m-1}}(\dd x)\otimes\PP(R^{(d,m)}_r\in\dd\rho)$
  is given by
  \begin{align}
    \La^{(d,m)}_r
    &:=\la\ka_d\cdot\frac{\lb\left(1-r^2;1+\frac{d-m}2,\frac m2\right)}{\cb\left(1+\frac{d-m}2,\frac m2\right)},%
    \label{eq:covering_intensity}
    \intertext{and, for every~$\rho>0$,}
    \PP(R^{(d,m)}_r>\rho)
    &:=\dfrac{\lb\left(1-r^2\cos^{-2}(a\rho);1+\frac{d-m}2,\frac m2\right)}{\lb(1-r^2;1+\frac{d-m}2,\frac m2)}
    \II_{\left\{\rho<\frac{\arccos r}a\right\}}\!.
    \label{eq:covering-radius}
  \end{align}
\end{lemma}
\begin{proof}
Indeed,
$\sf{d,m}\ge r$ if and only the sphere $r\dS^{m-1}$ is entirely covered by
the Voronoi flower associated with the projection of~$K^d_\la$ onto~$\RR^m$,
that is, by the spherical patches
\[\left(\frac{X'}2+\frac{\|X'\|}2\,\dB^m\right)\cap r\dS^{m-1}
=r B_{\dS^{m-1}}\left(\frac{X'}{\|X'\|},R_{X'}\right),\]
where $R_{X'}:=\arccos(\frac r{\|X'\|})$
and the points~$X'$
are the orthogonal projections of the points in~$\cP^d_\la\cap\cR^{d,m}_r$,
with
\[\cR^{d,m}_r:=\Bigl\{x\in\dB^d:x_1^2+\cdots+x_m^2\ge r^2\Bigr\};\]
see \cref{fig:m-dimensional-section}.
\begin{figure}[t]
    \centering
    \begin{tikzpicture}[scale=4]
        \begin{scope}
        \draw (0, 0) node {$\times$} node[above right] {$O$}
                (0, 0) circle (1)
                (30:1.2) node {$\dB^d\cap\RR^m$};
        \draw[blue,dashed]
                (0, 0) circle (.75)
                (0, 0) -- (273.55730976192072:.75) node[midway,right] {$r$};
        \draw[red,very thick]
                (240:.45) circle (.45);
        \end{scope}
        \begin{scope}
            \clip (240:.45) circle (.45);
            \fill[red!20,even odd rule] (0,0) circle (.73) circle (.77);
            \fill[pattern=north east lines,even odd rule] (0,0) circle (.73) circle (.77);
        \end{scope}
        \draw (0, 0) -- (240:.9) node[above,left=2pt,midway] {$\|X'\|$}
                (240:.9)
                node {$\bullet$} node[above=2pt,right=2pt] {$X'$}
                (240:.2) arc (240:273.55730976192072:.2) node[midway,below] {$\theta$};
    \end{tikzpicture}
    \caption{Projection onto~$\RR^m$.
    We have $\sf{d,m}\ge r$ if and only if the sphere~$r\dS^{m-1}$ (dashed)
    is covered by the union of its intersection with each petal (in red) of the Voronoi
    flower whose corresponding vertex lies in~$\cR^{d,m}_r$.
    Each projected vertex~$X'$ yields a
    geodesic ball of~$\dS^{m-1}$ (hatched) of radius $\theta:=\arccos(r/\lvert X'\rvert)$.}%
    \label{fig:m-dimensional-section}
\end{figure}
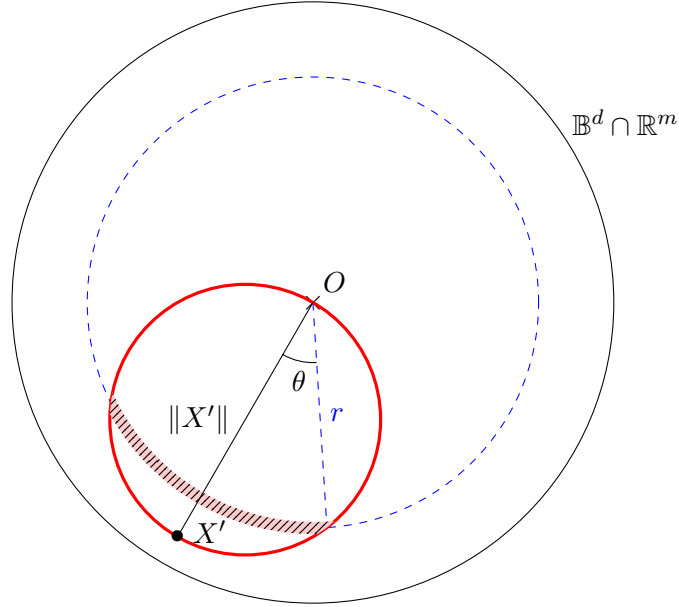
Considering a point~$X\in\cP^d_\la\cap\cR^{d,m}_r$,
the square of the norm of its orthogonal projection~$X'$
onto~$\RR^m$ has a law given by
\begin{equation}
    \dE\left[g\left({\lvert X'\rvert}^2\right)\right]=\dfrac{\int_{r^2}^1t^{\frac m2-1}{(1-t)}^{\frac{d-m}2}
        g(t)\,\dd t}{\lb\bigl(1-r^2;1+\tfrac{d-m}2,\tfrac m2\bigr)},
    \label{eq:norm-density-projected-vertices}
\end{equation}
for any measurable function $g\colon\co{0}{\infty}\to\co{0}{\infty}$.
We further note that the projections~$X'$ are identically distributed,
with~$X'/\|X'\|$ uniform on~$\dS^{m-1}$ (by rotational invariance).
We let $R^{(d,m)}_r$ denote a random variable with law
\begin{equation}
R^{(d,m)}_r\stackrel{(d)}=
\frac1aR_{X'}
=\frac1a\arccos{\frac r{\|X'\|}},\qquad\text{where}\quad a:=\frac1{\sqrt d}\,\frac{\sqrt{1-r^2}}{r}%
\label{eq:covering-radius-rayleigh}
\end{equation}
(this rescaling with the quantity~$a$ may seem arbitrary for the time being but is in fact designed for the future convergence in distribution in \cref{lem:covering-radius-convergence} and for the use of Janson's covering result in \cref{sec:janson-covering}).

We deduce that~$\sf{d,m}\ge r$ if and only if~$\dS^{m-1}$
is covered by the geodesic balls~$B_{\dS^{m-1}}(x,a\rho)$,
whose centers and radii~$(x,\rho)$
arise from a Poisson point process on~$\dS^{m-1}\times\oo{0}{\infty}$ with
intensity measure
\[\la\cdot|\cR^{d,m}_r|v_{\dS^{m-1}}(\dd x)\otimes\PP(R^{(d,m)}_r\in\dd\rho).\]
Now, the stated distribution function~\eqref{eq:covering-radius}
of~$R^{(d,m)}_r$ easily follows from~\eqref{eq:norm-density-projected-vertices}
and~\eqref{eq:covering-radius-rayleigh}:
\begin{align*}
\PP(R^{(d,m)}_r>\rho)
&=\PP\left({\|X'\|}^2>r^2\cos^{-2}(a\rho)\right)\\[.4em]
  &=\dfrac{\int_{r^2}^1t^{\frac m2-1}{(1-t)}^{\frac{d-m}2}
  \II_{\{t>r^2\cos^{-2}(a\rho)\}}\,\dd t}{\lb\bigl(1-r^2;1+\tfrac{d-m}2,\tfrac m2\bigr)}\\[.4em]
  &=\dfrac{\lb\left(1-r^2\cos^{-2}(a\rho);1+\frac{d-m}2,\frac m2\right)}{\lb(1-r^2;1+\frac{d-m}2,\frac m2)}
    \II_{\left\{\rho<\frac{\arccos r}a\right\}}\!.
\end{align*}
Furthermore,
\begin{align*}
\la\cdot\lvert\cR^{d,m}_r\rvert
&=\la\dotsint\II_{\{x_1^2+\cdots+x_m^2\ge r^2\}}
\ka_{d-m}{\left(1-x_1^2-\cdots-x_m^2\right)}^{\!\frac{d-m}2}\dd x_1\cdots\dd x_m\\[.4em]
&=\frac12\la m\ka_m\ka_{d-m}\int_{r^2}^1t^{\frac m2-1}{(1-t)}^{\frac{d-m}2}\dd t\\[.4em]
&=\la\ka_d\cdot\frac{\lb\left(1-r^2;1+\frac{d-m}2,\frac m2\right)}{\cb\left(1+\frac{d-m}2,\frac m2\right)}\\[.4em]
&=:\La^{(d,m)}_r,
\end{align*}
where the second equality comes from the use of spherical coordinates,
and the third equality is due to the expression~\eqref{eq:volume-ball}
for the volume of Euclidean balls and to the relation $\cb(a,b)=\Gamma(a)\Gamma(b)/\Gamma(a+b)$
between the beta and gamma functions.
\end{proof}
We are therefore reduced to understanding the probability
as~$d\to\infty$
of the covering event~\eqref{eq:covering-sphere}.
As we will see in the next section,
it turns out that Janson~\cite{Janson86} derived precise estimates
for the probability of covering a manifold of fixed dimension
using a Poisson process of i.i.d.\ patches whose intensity~$\La$
increases to infinity as the scale parameter~$a$ decreases to~$0$.
We conclude this section by showing that the random
radii~$R^{(d,m)}_r,\,r\in\oo{0}{1},$
have a common limit distribution as~$d\to\infty$.
This convergence will hold with respect
to the Monge-Kantorovitch-Wasserstein metric~$\Wd$,
which for two real random variables~$X$ and~$Y$ is given by
\[\Wd(X,Y)=\int_\RR\bigl|\PP(X>t)-\PP(Y>t)\bigr|\,\dd t\]
(see~\cite[Problem~2 p.\ 425]{Dudley02}).
We recall that the $\Wd$-convergence amounts to the convergence in distribution together with
the convergence of the first moment.
\begin{lemma}[Convergence of the patch radii]\label{lem:covering-radius-convergence}
    Let~$m\ge2$ be a fixed integer, let $r:=r(d)\in\oo{0}{1}$
    and let~$a$ as in~\eqref{eq:vanishing-parameter-a}.
    Suppose that
    \begin{equation}
      d\gg\frac1{r^2}+\log{\frac1{1-r}},\quad\text{or equivalently,}\enspace
      a+\frac1d\log{\frac1a}\to0.
      \label{eq:decay-vanishing-parameter}
    \end{equation}
    Then
    \begin{equation}
      \Wd\left(\log{R^{(d,m)}_r},\log R\right)=O\left(\frac1{dr^2}\log^2(dr^2)\right),\label{eq:W1-convergence-radius-rayleigh}
    \end{equation}
    where~$R$ is a standard Rayleigh random variable, with Lebesgue
    density $\rho\mapsto\rho\ee^{-\rho^2/2}$ on~$\oo{0}{\infty}$ and
    moments
    \begin{equation}
        \EE{R^k}=2^{\frac k2}\,\Gamma{\left(1+\frac k2\right)},
        \qquad k\in\ZZ_+.\label{eq:moments-rayleigh}
    \end{equation}
    Moreover, the convergence holds with the following upper bound:
    for every~$w>0$ and all~$d$ sufficiently large,
    \begin{equation}
      \sup_{\rho>0}\:\frac{\PP(R^{(d,m)}_r>\rho)}{\PP\left(\left[1+\frac w{\log{\frac1a}}\right]R>\rho\right)}\le1.\label{eq:uniform-rayleigh-upper-bound}
    \end{equation}
    \end{lemma}
\begin{proof}
  We observe from~\eqref{eq:covering-radius} that
  \begin{align*}
    \PP(R^{(d,m)}_r>\rho)
    &=\dfrac{\lb\left(1-r^2\cos^{-2}(a\rho);1+\frac{d-m}2,\frac m2\right)}{\lb(1-r^2;1+\frac{d-m}2,\frac m2)}
    \II_{\left\{\rho<\frac{\arccos r}a\right\}}
  \intertext{is a positive and continuously differentiable function of
  $x:=\tan^{-2}(a\rho)=\cos^{-2}(a\rho)-1$
  on $\oo{0}{r^{-2}-1}$.
  Letting $F(x):=\log{\lb(1-r^2-r^2\,x;1+\tfrac{d-m}2,\tfrac m2)}$
  allows us to write $\log{\PP(R^{(d,m)}_r>\rho)}=F(x)-F(0)$,
  so that by the mean value theorem there exists $\bar x\in\oo{0}{x}$ with}
  \log{\PP(R^{(d,m)}_r>\rho)}
  &=xF'(\bar x)\\[.4em]
  &=-r^2\,x\,\frac{{\bigl(1-r^2-r^2\bar x\bigr)}^{\frac{d-m}2}r^{m-2}{(1+\bar x)}^{\frac m2-1}}{\lb(1-r^2-r^2\,\bar x;1+\tfrac{d-m}2,\tfrac m2)}.
  \intertext{Applying \cref{lem:estimates-incomplete-beta-function}, part~1., for the denominator thanks to the assumption~\eqref{eq:decay-vanishing-parameter},
  we get (introducing $r^{-2}-1=da^2$)}
  \log{\PP(R^{(d,m)}_r>\rho)}
  &=-\left(1+\frac{d-m}2\right)\frac{r^2\,x}{1-r^2-r^2\,\bar x}
  \left[1+O\left(\frac{1-r^2-r^2\,\bar x}{dr^2(1+\bar x)}\right)\right]\\[.4em]
  &=-\frac x{2a^2}
  {\left(1-\frac{\bar x}{da^2}\right)}^{\!-1}
  \left[1+O\left(\frac 1{dr^2}\right)\right],
  \end{align*}
  where the error term is uniform in~$x\in\oo{0}{r^{-2}-1}$.
  Recalling $\bar x<x=\tan^2(a\rho)$
  and using the inequality $y<\tan y<y{(1-\frac 4{\pi^2}y^2)}^{-1}$
  (see, e.g.,~\cite{Becker78})
  with $y:=a\rho\in\oo{0}{\arccos r}\subset\oo{0}{\frac\pi2}$,
  we then obtain
  \begin{equation}
    {-{\frac{\rho^2}2{\left[{\left(1-\frac{4a^2\rho^2}{\pi^2}\right)}^{\!2}-\frac{\rho^2}d\right]}^{-1}
  \left[1+O\left(\frac1{dr^2}\right)\right]}}
  \le\log{\PP(R^{(d,m)}_r>\rho)}
  \le-\frac{\rho^2}2\left[1+O\left(\frac1{dr^2}\right)\right],%
  \label{eq:uniform-radius-estimate}
  \end{equation}
  uniformly in $0\le\rho<A$,
  where
  \[A:=\frac{\arccos r}a=r\sqrt d\,\frac{\arccos r}{\sqrt{1-r^2}}
    \asymp r\sqrt d\]
  tends to infinity by~\eqref{eq:decay-vanishing-parameter}.
  Recalling that the tail of the
  Rayleigh variable~$R$ fulfills $\log{\PP(R>\rho)}=-\frac{\rho^2}2$,
  the upper bound in~\eqref{eq:uniform-radius-estimate} entails that
  \begin{equation}
    \frac{\PP(R^{(d,m)}_r>\rho)}{\PP(R>\rho)}\II_{\left\{0\le\rho<A\right\}}
    =O(1).\label{eq:uniform-radius-estimate-1}
  \end{equation}
  Since
  \[{\left(1-\frac{4a^2\rho^2}{\pi^2}\right)}^{\!2}-\frac{\rho^2}d={\left[1+\rho^2\cdot O\left(a^2+\frac1d\right)\right]}^{\!2}
  ={\left[1+\rho^2\cdot O\left(\frac1{dr^2}\right)\right]}^{\!2},\]
  we also obtain from~\eqref{eq:uniform-radius-estimate} that,
  for any positive sequence $\eps:=\eps(d)\to0$,
  \begin{equation}
    \left|\PP(R^{(d,m)}_r>\rho)-\PP(R>\rho)\right|\II_{\left\{0\le\rho<\eps A\right\}}
    =\rho^2\cdot O\left(\eps^2\right).\label{eq:uniform-radius-estimate-2}
  \end{equation}
  Note that both error terms in~\eqref{eq:uniform-radius-estimate-1}
  and~\eqref{eq:uniform-radius-estimate-2} do not depend on~$\rho$.
  We deduce that
  \begin{align*}
    \Wd\left(\log{R^{(d,m)}_r},\log R\right)
    &=\int_\RR\left|\PP(\log R^{(d,m)}_r>x)-\PP(\log R>x)\right|\,\dd x\\[.4em]
    &=\int_0^\infty\left|\PP(R^{(d,m)}_r>\rho)-\PP(R>\rho)\right|\,\frac{\dd\rho}\rho\\[.4em]
    &\le\int_0^{\eps A}\rho^2\cdot O(\eps^2)\,\frac{\dd\rho}\rho
    +\int_{\eps A}^\infty{\ee^{-\frac{\rho^2}2}}\cdot{O(1)}\,\frac{\dd\rho}\rho\\[.4em]
    &=O(\eps^4A^2)+O\left(\frac{\ee^{-\frac12\eps^2 A^2}}{\eps^2A^2}\right).
  \end{align*}
  Choosing $\eps:=2\sqrt{\log A}/A$, we obtain
  \[\Wd\left(\log{R^{(d,m)}_r},\log R\right)=O\left(\frac{\log^2 A}{A^2}\right)
  =O\left(\frac1{dr^2}\log^2(dr^2)\right),\]
  as stated. Finally, for~$w>0$,
  the upper bound in~\eqref{eq:uniform-radius-estimate}
  yields that, uniformly for all~$\rho>0$,
  \begin{align*}
    \frac{\PP(R^{(d,m)}_r>\rho)}{\PP\left(\left[1+\frac w{\log{\frac1a}}\right]R>\rho\right)}
    &\le\exp\left(-\frac{\rho^2}2\left[1-{\left(1+\frac w{\log{\frac1a}}\right)}^{\!-2}+O\left(\frac1{dr^2}\right)\right]\right)\\[.4em]
    &=\exp\left(-\frac{w\rho^2}{\log{\frac1a}}\bigl[1+o(1)\bigr]\right)\\[.4em]
    &\le1,
  \end{align*}
  provided~$d$ is chosen sufficiently large,
  where we used that
  $\frac1{dr^2}\log{\frac1a}=(a^2+\frac1d)\log{\frac1a}\to0$, by~\eqref{eq:decay-vanishing-parameter}.

\noindent
  For the moments of the standard Rayleigh distribution, see, e.g.,~\cite[§~18.3]{Johnson94}.
\end{proof}

\subsection{Application of Janson's covering result.}\label{sec:janson-covering}
In~\cite{Janson86}, Janson showed that the
number of random ``small sets''
needed to cover a fixed ``big set'' converges%
, when properly normalized,
to the Gumbel extreme value distribution as the ``size'' of the small sets
tends to zero.
More precisely, let the big set be a~$\cC^2$, $D$-dimensional compact Riemannian manifold~$M$
with volume measure~$v_M$,
and suppose that the small sets are i.i.d.\ geodesic balls,
that is, they are all of the form
$B_M(x_i,a\rho_i),\,i\ge1,$
where~$a>0$ is a vanishing scale parameter,
and the centers and radii~$(x_i,\rho_i)$ arise as the atoms of
a Poisson point process on~$M\times\oo{0}{\infty}$
with intensity $\La v_M(\dd x)\otimes\PP(R\in\dd\rho)$,
for some positive random variable~$R$
and $\La:=\La(a)\to\infty$ as~$a\to0$.
Then,
denoting by $\Cover(\La,R,a;M)$ the event
\[M=\bigcup_{i\ge1}B_M(x_i,a\rho_i),\]
Janson~\cite[Lemma~8.1]{Janson86} proved that
\begin{equation}
    \lim_{a\to0}\PP\bigl(\Cover(\La,R,a;M)\bigr)
    =\ee^{-\ee^{-\tau}},\tag{$J$}
\label{eq:janson-gumbel}
\end{equation}
under the following two conditions:
\begin{gather}
\EE{R^q}<\infty\quad\text{for some~$q>D$},
\tag{$J_1$}
\label{eq:moment-janson}
\intertext{and}
J(\La,R,a;M):=ba^Dv_M(M)\,\La-\log{\frac1{ba^D}}-D\log{\log{\frac1{ba^D}}}
-\log{\alpha(R)}\:\xrightarrow[a\to0]{}\:\tau\in\RR,\tag{$J_2$}
\label{eq:condition-janson}
\end{gather}
with
\begin{align}
b:=b(R;M)&:=\frac{\pi^{D/2}\EE{R^D}}{\Gamma(1+\frac D2)v_M(M)},\tag{$J_b$}
\label{eq:janson-b}
\intertext{and~\cite[Eq.\ (9.24)]{Janson86}}
\alpha(R)&:=\frac1{D!}{\left(
\frac{\sqrt\pi\,\Gamma(1+\frac D2)}{\Gamma(\frac{D+1}2)}
\right)}^{\!D-1}
\frac{{\left(\EE{R^{D-1}}\right)}^D}{{\left(\EE{R^D}\right)}^{D-1}}.%
\label{eq:janson-alpha}\tag{$J_\alpha$}
\end{align}

In view of \cref{lem:reduction-to-covering-problem}, we would then like to estimate the probability
\begin{equation}
    \PP\left(\sf{d,m}\ge r\right)=
    \PP\left(\Cover(\La^{(d,m)}_r,R^{(d,m)}_r,a;\dS^{m-1})\right),
    \label{eq:covering-probability-rayleigh}
\end{equation}
where the scale parameter $a:=\sqrt{1-r^2}/r\sqrt d$ vanishes as~$d\to\infty$.
Although Janson's original result
is stated only when the distribution of the random radius does
not depend on~$a$, we show in \cref{lem:application-extension-janson} below, under a minor reinforcement of the
assumption given at~\eqref{eq:decay-vanishing-parameter}, that it still holds
when we use the radii~$R^{(d,m)}_r$ instead of their Rayleigh limit~$R$ in \cref{lem:covering-radius-convergence}.
This is done thanks to a slight improvement of Janson's result,
\cref{pro:extension-janson}, which may be of independent interest (see~\hyperref[sec:appendix]{Appendix}).
\begin{lemma}[Application of the extension of Janson's result]\label{lem:application-extension-janson}
  With the notation of~\cref{lem:reduction-to-covering-problem,lem:covering-radius-convergence},
  suppose~\eqref{eq:condition-janson} holds with $\La:=\La^{(d,m)}_r$ and $M:=\dS^{m-1}$
  (i.e., $J\bigl(\La^{(d,m)}_r,R,a;\dS^{m-1}\bigr)\to\tau$),
  and suppose also that
  \begin{equation}
    a+\frac1d\log^{2+\ga}{\frac1a}\:\xrightarrow[d\to\infty]{}\:0
    \label{eq:decay-vanishing-parameter-2}
  \end{equation}
  holds for some $\ga\in\oo{0}{\frac12}$.
  Then
  \[\lim_{d\to\infty}\PP\bigl(\Cover(\La^{(d,m)}_r,R^{(d,m)}_r,a;\dS^{m-1})\bigr)
    =\ee^{-\ee^{-\tau}}\!.\]
\end{lemma}
\begin{proof}
  We apply \cref{pro:extension-janson} in~\hyperref[sec:appendix]{Appendix}.
  We have~\eqref{eq:condition-janson} for $\La:=\La^{(d,m)}$ and $D:=m-1$,
  and also~\eqref{eq:moment-janson}
  because all moments of the Rayleigh distribution are finite.
  It remains to show that the two conditions of \cref{pro:extension-janson}
  related to~$R_a:=R^{(d,m)}_r$ are satisfied.
  The first one is~\eqref{eq:uniform-rayleigh-upper-bound}
  given by \cref{lem:covering-radius-convergence}.
  For the second one,~\eqref{eq:W1-convergence-radius-rayleigh} gives
  \begin{equation}
    \Wd\left(\log{R^{(d,m)}_r},\log R\right)=O\left(\frac1{dr^2}\log^2(dr^2)\right)
    =o\left(\frac1{\log^2{\frac1a}}\right),
  \end{equation}
  since, by~\eqref{eq:decay-vanishing-parameter-2} and the relation
  $\frac1{dr^2}=\frac1d+a^2$,
  \[\frac1{dr^2}\log^2(dr^2)\log^2{\frac1a}
  ={\left[\left(\frac1{dr^2}\right)\log^{\frac{4+2\ga}\ga}(dr^2)\right]}^{\!\frac\ga{2+\ga}}
  \cdot{\left[\frac1d\log^{2+\ga}{\frac1a}
  +a^2\log^{2+\ga}\frac1a\right]}^{\!\frac2{2+\ga}}
  \:\xrightarrow[d\to\infty]{}\:0.\qedhere\]
\end{proof}
The applicability of \cref{lem:application-extension-janson}
is done in \cref{lem:condition_janson} below.
Plugging in the expressions of~$\EE{R^k}$ given by ~\eqref{eq:moments-rayleigh}
into the expression~\eqref{eq:janson-alpha} on page~\pageref{eq:janson-alpha},
we can see that~$\alpha(R)$ simplifies to
\begin{align}
    \alpha:=\alpha(R)&=\frac{\pi^{\frac{m-1}2}\Gamma(\frac{m+1}2)}{(m-1)!}%
    =\frac{\pi^{\frac m2}}{2^{m-1}\,\Gamma(\frac m2)},%
    \label{eq:rayleigh-alpha}
    \intertext{by an application of Legendre's duplication formula.
    Similarly, using also the expression of~$\ka_m$ in~\eqref{eq:volume-ball}
    and the fundamental property $\Gamma(1+\frac m2)=\frac m2\,\Gamma(\frac m2)$,
    the expression of~$b(R;M)$ in~\eqref{eq:janson-b} reduces to}
    b:=b(R;\dS^{m-1})&=\frac{{(\sqrt{2\pi})}^{m-1}\,\Gamma(\frac m2)}{2\pi^{\frac m2}}
    =\frac{2^{\frac{m-3}2}\,\Gamma(\frac m2)}{\sqrt\pi}.\label{eq:rayleigh-b}
\end{align}
\begin{lemma}[Verification of Janson's condition~\eqref{eq:condition-janson}
and of~\eqref{eq:decay-vanishing-parameter-2}]\label{lem:condition_janson}
  Let~$m\ge2$ be a fixed integer,
  and suppose that one of the three assumptions~\eqref{hyp:subcritical},~\eqref{hyp:critical}, or~\eqref{hyp:supercritical} below occurs:
  \begin{gather}
      \log d\ll\log{\la\ka_d}\ll d,%
      \tag{$\rA_{\textrm{sub}}$}\label{hyp:subcritical}\\[.4em]
      \log{\la\ka_d}\sim dx\ \text{with}\ x\in\oo{0}{\infty},%
      \tag{$\rA_{\textrm{crit}}$}\label{hyp:critical}\\[.4em]
      d\ll\log{\la\ka_d}\ll d^{\frac32-\ga}\text{ for some $\ga\in\oo[scaled]{0}{\frac12}$}.\tag{$\rA_{\textrm{sup}}$}\label{hyp:supercritical}
  \end{gather}
  Let~$\alpha$ and~$b$ as in~\eqref{eq:rayleigh-alpha}
  and~\eqref{eq:rayleigh-b}.
  Then for every~$\tau\in\RR$, there exists $r:=r(d;\tau)>0$
  such that, for $a:=a(d;\tau)$ and~$\La^{(d,m)}_r$ as in~\eqref{eq:vanishing-parameter-a}
  and~\eqref{eq:covering_intensity}, condition~\eqref{eq:condition-janson} holds:
  \begin{equation}
    J(\La^{(d,m)}_r,R,a;\dS^{m-1})=
      ba^{m-1}\,m\ka_m\La^{(d,m)}_r
      +\log\left(ba^{m-1}\right)
      -(m-1)\log\left[-{\log\left(ba^{m-1}\right)}\right]
      -\log\alpha
      \to\tau.%
      \label{eq:condition-janson-rayleigh}
  \end{equation}
  Furthermore,~\eqref{eq:decay-vanishing-parameter-2} holds: $a+\frac1d\log^{2+\ga}{\frac1a}\to0$ for some $\ga\in\oo{0}{\frac12}$.
\end{lemma}
\begin{proof}
    According to~\eqref{eq:vanishing-parameter-a},
    we seek $r:={(1+da^2)}^{-\frac12}\in\oo{0}{1}$.
    We start with the following asymptotics of~\eqref{eq:covering_intensity}, obtained by \cref{lem:estimates-incomplete-beta-function}, part~2.,
    \begin{align}
        \La^{(d,m)}_r
        &=\la\ka_d\,
            \frac{{(1-r^2)}^{1+\frac{d-m}2}r^{m-2}{(\frac d2)}^{\frac m2-1}}{\Gamma(\frac m2)}
            \left[1+O\left(\frac1{dr^2}\right)\right]\notag\\[.4em]
        &=\la\ka_d\,d^{\frac d2}\,\frac{a^{d+2-m}{(1+da^2)}^{-\frac d2}}{2^{\frac m2-1}\,\Gamma(\frac m2)}
        \left[1+O\left(\frac1d+a^2\right)\right],\notag\end{align}
    provided that $\frac1{dr^2}=\frac1d+a^2\ll1$.
    In this case,
    \begin{align}
        ba^{m-1}\,m\ka_m\,\La^{(d,m)}_r
        &=\frac{\sqrt2\,\pi^{\frac{m-1}2}}{\Gamma(\frac m2)}\,\la\ka_d\,\frac a{{\left(1+\frac1{da^2}\right)}^{d/2}}
        \left[1+O\left(\frac1d+a^2\right)\right],\label{eq:condition-janson-1st-term}
    \intertext{Next,}
        \log\left(ba^{m-1}\right)
        &=-(m-1)\log{\frac1a}+\log b,\label{eq:condition-janson-2nd-term}
    \intertext{%
    and}
        {-{(m-1)\log\left[-{\log\left(ba^{m-1}\right)}\right]}}
        &=-{(m-1)\log{\log{\frac1a}}}
        -(m-1)\log(m-1)+o(1).\label{eq:condition-janson-3rd-term}
    \end{align}
    Adding~\eqref{eq:condition-janson-1st-term},~\eqref{eq:condition-janson-2nd-term},~\eqref{eq:condition-janson-3rd-term}
    and~$-{\log\alpha}$
    yields
    \begin{align}
    &ba^{m-1}\,m\ka_m\La^{(d,m)}_r
    +\log\left(ba^{m-1}\right)
    -(m-1)\log\left[-{\log\left(ba^{m-1}\right)}\right]
    -\log\alpha
    \label{eq:condition-janson-lhs}\\[.4em]
    &\quad=
    \frac{\sqrt2\,\pi^{\frac{m-1}2}}{\Gamma(\frac m2)}\,\la\ka_d\,\frac a{{\left(1+\frac1{da^2}\right)}^{d/2}}
    \left[1+O\left(\frac1d+a^2\right)\right]-(m-1)\log{\frac1a}-{(m-1)\log{\log{\frac 1a}}}
    -\log{B_m}+o(1),\label{eq:condition-janson-rhs}
    \end{align}
    where
    \begin{equation}
        B_m:=\frac{\alpha\,{(m-1)}^{m-1}}b
        =\frac{\pi^{\frac{m+1}2}{(m-1)}^{m-1}}{2^{\frac{3m-5}2}{\Gamma(\frac m2)}^2}\label{eq:constant-Bm}
    \end{equation}
    follows from the expressions of~$\alpha$ in~\eqref{eq:rayleigh-alpha}
    and~$b$ in~\eqref{eq:rayleigh-b}.
    We now observe by monotonicity and continuity
    in the variable~$a\in\oo{0}{1}$ that the equation
    \begin{equation}
      \frac{\sqrt2\,\pi^{\frac{m-1}2}}{\Gamma(\frac m2)}\,\la\ka_d\,\frac a{{\left(1+\frac1{da^2}\right)}^{d/2}}
     =(m-1)\log{\frac1a}+{(m-1)\log{\log{\frac 1a}}}
     +\log{B_m}+\tau\label{eq:implicit-definition-vanishing-parameter-a}
    \end{equation}
    has a unique solution $a:=a(d;\tau)$. Moreover, comparing both sides
    of this equation, we observe that
    any accumulation point of~$a(d;\tau)$ must equal~$0$
    since~$\la\ka_d\to\infty$. Consequently, $a(d;\tau)\to0$ and we henceforth make this choice for the parameter~$a$ (where, for sake of clarity, we omit the dependency
    in~$d$ and~$\tau$).
    If we plug in~\eqref{eq:implicit-definition-vanishing-parameter-a} into~\eqref{eq:condition-janson-rhs}, we obtain that~\eqref{eq:condition-janson-lhs}
    reduces to
    \begin{equation}
    ba^{m-1}\,m\ka_m\La^{(d,m)}_r
    +\log\left(ba^{m-1}\right)
    -(m-1)\log\left[-{\log\left(ba^{m-1}\right)}\right]-\log\alpha=\tau+O\left(\left[\frac1d+a^2\right]\cdot\log{\frac1a}\right),%
    \label{eq:condition-janson-rhs-2}
    \end{equation}
    Note that, in order to get~\eqref{eq:condition-janson-rayleigh}, it is enough to justify that the last error term in the right-hand side of~\eqref{eq:condition-janson-rhs-2} goes to zero. Actually, that last term will be negligible as soon as~\eqref{eq:decay-vanishing-parameter}
    and a fortiori~\eqref{eq:decay-vanishing-parameter-2} is satisfied. The remainder of the proof is then devoted to asymptotically relate~%
$\frac1a$ and $\log{\la\ka_d}$
    in order to check that~\eqref{eq:decay-vanishing-parameter-2} holds in all three regimes.
    
    To begin with, passing to the
    logarithm in~\eqref{eq:implicit-definition-vanishing-parameter-a} easily yields
    \begin{equation*}
      \frac d2\log\left(1+\frac1{da^2}\right)
      =\log{\frac{A_m\,\la\ka_d}{\frac1a\log{\frac1a}}}
      -\frac{(m-1)\log{\log{\frac1a}}+\log{B_m}+\tau}{(m-1)\log{\frac1a}}
      +o\left(\frac{\log{\log{\frac1a}}}{\log{\frac1a}}\right),
    \end{equation*}
    where
    \begin{equation}
      A_m:=\frac{\sqrt2\,\pi^{\frac{m-1}2}}{(m-1)\Gamma(\frac m2)}.\label{eq:constant-Am}
    \end{equation}
    Dropping the~$O(1)$ terms and dividing both sides by~$d$, we have in particular
    \begin{equation}
        \frac1d\log{\frac1a}
        +\frac1d\log{\log{\frac1a}}
        +O\left(\frac1d\right)
        =\frac1d\log{\la\ka_d}-\frac12\log\left(1+\frac1{da^2}\right).\label{eq:asymptotic-relation-a-Lambda}
    \end{equation}
    We start with the subcritical case and assume that~\eqref{hyp:subcritical} holds.
    Since~$\log{\la\ka_d}\ll d$, the left-hand side of~\eqref{eq:asymptotic-relation-a-Lambda} is equivalent to
    $\frac1d\log{\frac1a}$ (because~$a\to0$), while the right-hand side is bounded from above by $\frac1d\log{\la\ka_d}$. This entails that
    $\frac1d\log{\frac1a}\to0$,
    and then, plugging this back into~\eqref{eq:asymptotic-relation-a-Lambda}, $\log(1+\frac1{da^2})\to0$.
    Thus~$da^2\to\infty$, so $\log{\frac1a}=O(\log d)$,
    and~\eqref{eq:asymptotic-relation-a-Lambda}
    becomes
    \[\frac1d\log{\la\ka_d}-\frac1{2da^2}\bigl(1+o(1)\bigr)
    =O\left(\frac{\log d}d\right).\]
    Because $\log d\ll\log{\la\ka_d}$, this implies
    \begin{align*}
      &&\frac1a&\sim\sqrt{2\log{\la\ka_d}}
        &&\text{(subcritical regime)},
    \intertext{which results in $\frac1d\log^2{\frac1a}\to0$
    (so~\eqref{eq:decay-vanishing-parameter-2} holds for any $\ga\in\oo{0}{\frac12}$).}
    \intertext{%
    Next, for the critical case, let us now assume~\eqref{hyp:critical}, i.e.,
    $\frac1d\log{\la\ka_d}\to x\in\oo{0}{\infty}$.
    Then~\eqref{eq:asymptotic-relation-a-Lambda}
    implies that $\log{\frac1a}=O(d)$,
    which in turn implies that $\log\bigl(1+\frac1{da^2}\bigr)=O(1)$,
    i.e., $a\gtrsim\frac1{\sqrt d}$. Hence, by~\eqref{eq:asymptotic-relation-a-Lambda} again,
    $\frac1d\log{\la\ka_d}-\frac12\log\bigl(1+\frac1{da^2}\bigr)\to0$,
    or in other words,}
    &&\frac1a&\sim\sqrt{d\left(\ee^{2x}-1\right)}
    &&\text{(critical regime)},
    \intertext{which is stronger than~\eqref{eq:decay-vanishing-parameter-2}.}
    \intertext{Lastly, we consider the supercritical case and assume~\eqref{hyp:supercritical}. Then the inequality $\frac1d\log{\frac1a}\le\frac12\log(1+\frac1{da^2})$ (which follows from the concavity of the logarithm) and~\eqref{eq:asymptotic-relation-a-Lambda} give
    \[\frac1d\log{\la\ka_d}=O\left(\log\Bigl(1+\frac1{da^2}\Bigr)\right).\]
    Because $\log{\la\ka_d}\gg d$, this forces~$da^2\to0$, so~\eqref{eq:asymptotic-relation-a-Lambda} produces
    \[\bigl(1+o(1)\bigr)\log{\frac1a}
    =\log\left({(\la\ka_d)}^{\frac1d}\sqrt d\right)
    +o(1),\]
    that is}
&&\frac1a&\sim{(\la\ka_d)}^{\frac1d}\sqrt d
    &&\text{(supercritical regime)}.
    \end{align*}
    This will imply $\frac1d\log^{2+\ga}{\frac1a}\to0$ and~\eqref{eq:decay-vanishing-parameter-2} (for some $\ga\in\oo{0}{\frac12}$)
    if we further assume that $\log{\la\ka_d}\ll d^{\frac32-\ga}$,
    as stated in~\eqref{hyp:supercritical}.

    Finally, we can see that the error term
    in~\eqref{eq:condition-janson-rhs-2} goes to zero because~\eqref{eq:decay-vanishing-parameter-2}
    and a fortiori~\eqref{eq:decay-vanishing-parameter} hold under each of
    the three stated conditions~\eqref{hyp:subcritical},~\eqref{hyp:critical} and~\eqref{hyp:supercritical}.
\end{proof}
\begin{scholium}[Summary of useful asymptotics derived in the proof of \cref*{lem:condition_janson}]
  The sequence~$a$ fulfills the implicit asymptotic equality
  \begin{equation}
    \frac d2\log\left(1+\frac1{da^2}\right)
    =\log{\frac{A_m\,\la\ka_d}{\frac1a\log{\frac1a}}}
    -\frac{(m-1)\log{\log{\frac1a}}+\log{B_m}+\tau}{(m-1)\log{\frac1a}}
    +o\left(\frac1{\log{\frac1a}}\right),\label{eq:asymptotics-vanishing-parameter}
  \end{equation}
  where we record
  \begin{equation}
      A_m:=\frac{\sqrt2\,\pi^{\frac{m-1}2}}{(m-1)\Gamma(\frac m2)},\qquad
      B_m:=\frac{\pi^{\frac{m+1}2}{(m-1)}^{m-1}}{2^{\frac{3m-5}2}{\Gamma(\frac m2)}^2},\label{eq:constants-Am-and-Bm}
  \end{equation}
  from~\eqref{eq:constant-Am} and~\eqref{eq:constant-Bm}.
  Furthermore we have, in all regimes,
  \begin{equation}
      \frac1a\sim\sqrt{d\left({(\la\ka_d)}^{\frac2d}-1\right)}
      \sim\begin{cases}
        \sqrt{2\log{\la\ka_d}},&\text{under~\eqref{hyp:subcritical}},\\[.4em]
        \sqrt{d(\ee^{2x}-1)},&\text{under~\eqref{hyp:critical}},\\[.4em]
        {(\la\ka_d)}^{\frac1d}\sqrt d,&\text{under~\eqref{hyp:supercritical}}.
      \end{cases}
  \label{eq:asymptotics-inverse-vanishing-parameter}
  \end{equation}
\end{scholium}

We are now ready to prove \cref{thm:main,cor:regimes-multidirectional}.
\begin{proof}[Proof of \cref*{thm:main}]
    For~$\tau\in\RR$, $a:=a(d;\tau)$,
    and~$r:={(1+da^2)}^{-\frac12}$ as in
    \cref{lem:condition_janson},
    the conditions to apply \cref{lem:application-extension-janson}
    are in place, hence (recalling~\eqref{eq:covering-probability-rayleigh})
    \[\lim_{d\to\infty}\PP(\sf{d,m}\ge r)
    =\ee^{-\ee^{-\tau}}\!.\]
    Now, notice that
    \begin{align*}
        \PP(\sf{d,m}\ge r)
        &=\PP\left(\sf{d,m}\ge{(1+da^2)}^{-\frac12}\right)\\[.4em]
        &=\PP\left[d\log{\frac1{\sqrt{1-{\bigl(\sf{d,m}\bigr)}^2}}}
        \ge\frac d2\log\left(1+\frac1{da^2}\right)\right].
    \intertext{Plugging in~\eqref{eq:asymptotics-vanishing-parameter}, we find}
        \PP(\sf{d,m}\ge r)
        &=\PP\Biggl[
            \log{\frac{A_m\,\la\ka_d}{\frac1a\log{\frac1a}}}
            -d\log{\frac1{\sqrt{1-{\bigl(\sf{d,m}\bigr)}^2}}}
            \le\frac{(m-1)\log{\log{\frac1a}}+\log{B_m}+\tau}{(m-1)\log{\frac1a}}+o\left(\frac1{\log{\frac1a}}\right)
        \Biggr].
    \end{align*}
    To conclude, it remains to express $\log{\frac1a}$ and
    $\log{\log{\frac1a}}$ in terms of $\mathfrak s(d)=\log{\sqrt{d\left({(\la\ka_d)}^{\frac2d}-1\right)}}$. According to~\eqref{eq:asymptotics-inverse-vanishing-parameter},
    \[
        \log{\frac1a}=\mathfrak s(d)+o(1),\qquad\text{and also}\quad
        \log{\log{\frac1a}}=\log{\mathfrak s(d)}+o(1).
    \]
    Hence
    \[
    (m-1)d\,\mathfrak s(d)
    \left[\frac1d\log{\frac{A_m\,\la\ka_d}{\ee^{\mathfrak s(d)}\mathfrak s(d)}}
        -\log{\frac1{\sqrt{1-{\bigl(\sf{d,m}\bigr)}^2}}}\right]
    -(m-1)\log{\mathfrak s(d)}-\log{B_m}
    \]
    converges in law to the standard Gumbel distribution.
    This establishes \cref{thm:main} with
    \begin{align}
      \mathfrak a(d;m)&:=(m-1)\mathfrak s(d)
      \log{\frac{A_m\,\la\ka_d}{\mathfrak s(d)}}
      -(m-1){\mathfrak s(d)}^2
      -(m-1)\log{\mathfrak s(d)}-\log{B_m},\label{eq:a-multidirectional}
      \intertext{and}
      \mathfrak b(d;m)
      &:=(m-1)d\,\mathfrak s(d),\label{eq:b-multidirectional}
      \end{align}
    where we recall that~$A_m$ and~$B_m$ have been defined at~\eqref{eq:constants-Am-and-Bm}.
\end{proof}
\begin{proof}[Proof of \cref*{cor:regimes-multidirectional}]
  We get from \cref{thm:main} that
  \[\mathfrak a(d;m)-\mathfrak b(d;m)\log{\frac1{\sqrt{1-{\bigl(\sf{d,m}\bigr)}^2}}=O_\dP(1)},\]
  where $X(d)=O_{\dP}(1)$ means that
  $\lim_{A\to\infty}\limsup_{d\to\infty}\PP(|X(d)|>A)=0$.
  Multiplying by $2{\mathfrak b(d;m)}^{-1}$ and taking the exponential function,
  it follows that
  \[\ee^{\frac{2\mathfrak a(d;m)}{\mathfrak b(d;m)}}\left(1-{\bigl(\sf{d,m}\bigr)}^2\right)
  =1+\frac1{\mathfrak b(d;m)}\cdot O_\dP(1),\]
  where $\ee^{\frac{2\mathfrak a(d;m)}{\mathfrak b(d;m)}}\sim{(\la\ka_d)}^{\frac2d}$
  from~\eqref{eq:a-multidirectional} and~\eqref{eq:b-multidirectional}.
  Since $\log{\la\ka_d}\ll d^2$, this is also equivalent to ${(\la\ka_d)}^{\frac2{d+1}}$.
  We finish the proof similarly to that of \cref{thm:regimes}.
\end{proof}

\section{Consequences on the radius-vector function}\label{sec:radius-vector-function}
To conclude this work, we prove an easy
consequence on the asymptotics of the radius-vector function~$\rvf$ given at~\eqref{eq:radius-vector-function}.
Recall that~$\rvf\le\sf$, see \cref{fig:polytope}.

\begin{corollary}[Subcritical and supercritical regimes for~$\rvf$]\label{cor:radius-vector-function} Let~\cref{eq:assumption-origin} hold.
    \begin{enumerate}[label=(\roman*)]%
        \item In the subcritical regime, under $\log{\la\ka_d}\ll d$,
        \[\limsup_{d\to\infty}\sqrt\frac d{2\log{\la\ka_d}}\,\rvf\le1,\quad\text{in probability}.\]
        \item In the supercritical regime, under the condition $\log{\la\ka_d}\gg d$
        \[\lim_{d\to\infty}(\la\ka_d)^{\frac2{d+1}}\Bigl(1-\rvf\Bigr)=\frac12,\quad\text{in probability}.\]
    \end{enumerate}
\end{corollary}
\begin{proof}
    Since~$\rvf\le\sf$, (i) is an immediate
    consequence of \cref{thm:regimes},
    and~(ii) will also follow from \cref{thm:regimes} if we prove
    that, for every $\eps\in\oo{0}{1}$,
    \begin{equation}
        \PP\left({\bigl(\la\ka_d\bigr)}^{\frac2{d+1}}\Bigl(1-\rvf(u)\Bigr)>\frac12+\eps\right)\xrightarrow[d\to\infty]{}0.\label{eq:upper-bound-rvf}
    \end{equation}
    Recall the definition of the cap $\cC^{d}(r;u):=\{x\in\dB^d:\langle u,x\rangle>r\}$.
    Let
    $\{X_1,\ldots,X_N\}=\cP^d_\la\cap\cC^d(r;u)$ and denote
    by $X'_i,\,1\le i\le n,$ their projections
    onto the $(d-1)$-dimensional hyperplane $\{x\in\RR^d:\langle u,x\rangle=r\}$,
    with $r\in\oo{0}{1}$ arbitrary.
    The number~$N$ has a Poisson distribution with parameter
    $\ell(r):=\la\lvert\cC^d(r;u)\rvert$, and conditional on~$N$,
    the points~$X'_i-ru$ are i.i.d.\ according to a symmetric distribution
    on~$\RR^{d-1}$. First, Wendel's formula~\cite{Wendel62} allows us to write
    \begin{align*}
        \PP\Bigl(0\notin\hull\{X'_1-ru,\ldots,X'_N-ru\}\:\Big|\:N\Bigr)
        &=\II_{\{N<d\}}
        +\II_{\{N\ge d\}}2^{-(N-1)}\sum_{k=0}^{d-2}\binom{N-1}k\\[.4em]
        &=1-\II_{\{N\ge d\}}2^{-(N-1)}\sum_{k=d}^N\binom{N-1}{k-1}.
    \end{align*}
\begin{figure}[t]
    \centering\includegraphics[width=.5\textwidth]{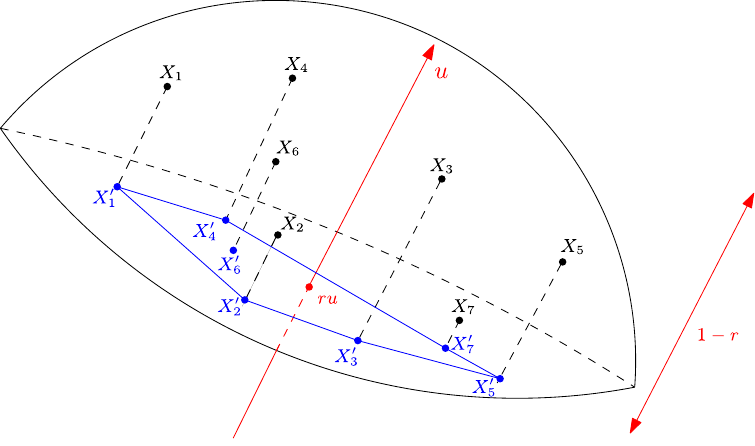}
   \caption{If $ru\in\hull(X'_1,\ldots,X'_N)$, then~$\rvf(u)>r$.}%
    \label{fig:radius-vector}
\end{figure}
    Next, we observe that $ru\notin\hull\{X'_1,\ldots,X'_N\}$
    on the event $\{\rvf(u)\le r\}$ (see \cref{fig:radius-vector}), so we get
    \begin{align}
        \PP\Bigl(\rvf(u)\le r\Bigr)
        &\le
        \EE\left[\PP\Bigl(0\notin\hull\{X'_1-ru,\ldots,X'_N-ru\}\:\Big|
        \:N\Bigr)\right]\notag\\[.4em]
        &=1-\EE\left[\II_{\{N\ge d\}}2^{-(N-1)}\sum_{k=d}^N\binom{N-1}{k-1}\right]\notag\\[.4em]
        &=1-\PP(S_{N-1}\ge d),\label{eq:tail-radius-vector-function}
    \end{align}
    where as in the proof of \cref{lem:origin}, the random variable $S_{N-1}$ conditional on~$N$
    has the Binomial$(N-1,\frac12)$ distribution.
    We now let~$r$ depend on~$d$ and observe that, for
    $r:=1-{(\la\ka_d)}^{-\frac2{d+1}}(\frac12+\eps)$,
    we have
    \[\rvf(u)\le r\iff{(\la\ka_d)}^{\frac2{d+1}}\Bigl(1-\rvf(u)\Bigr)>\frac12+\eps,\]
    so in order to get~\eqref{eq:upper-bound-rvf} it suffices to prove that the upper bound in~\eqref{eq:tail-radius-vector-function} goes to~$0$, where~$N$ is a Poisson r.v.\
    with mean~$\ell(r)$.
    But
    \begin{align*}
        \ell(r)
        &=\frac{\la\ka_{d-1}}2
    \lb\left(1-r^2;\tfrac{d+1}2,\tfrac12\right)\\[.4em]
    &=-{\log{\PP\left(\sf\le r\right)}}\\[.4em]
    &=\exp\left\{\log{\la\ka_d}+\frac{d+1}2\log(1-r^2)
    -\log{r\sqrt{2\pi d}}+o(1)\right\},
    \end{align*}
    by~\eqref{eq:volume-spherical-cap},~\eqref{eq:cdf-support-function}, and~\eqref{eq:exponent-cdf-support-function-2}, provided that~$d\gg r^{-2}$.
    In fact,
    \begin{align*}
        1-r^2&={(\la\ka_d)}^{-\frac2{d+1}}(1+2\eps)\left[1+O\left({(\la\ka_d)}^{-\frac2{d+1}}\right)\right],
    \intertext{so}
        \frac{d+1}2\log(1-r^2)&=-{\log{\la\ka_d}}+\frac{d+1}2\log(1+2\eps)+o(d).
    \end{align*}
    We can then see that
    $\ell(r)={(1+2\eps)}^{\frac d2+o(d)}\gg d$ and  so~$N\gg d$, w.h.p. Hence
    $\PP(S_{N-1}\ge d)\to1$ and this completes the proof.
\end{proof}
\begin{remark}[Asymptotics for the volume ratio]
By integrating the $d$-th power of the radius-vector function~$\rvf$ over the whole sphere~$\dS^{d-1}$,
we obtain the $d$-dimensional volume of $K^d_\la$, namely
\begin{align}
  \lvert{K^d_\la\rvert}&=\int_{\dS^{d-1}}\sigma_{d-1}(\dd u)\int_0^{\rvf(u)} t^{d-1}\,\dd t,\notag
\intertext{which by Fubini's theorem leads to}
  \EE{\lvert{K^d_\la}\lvert}&=\ka_d\int_0^{+\infty}\PP\left(\rvf>t\right) \,dt^{d-1}\,\dd t=\ka_d\EE{{\bigl(\rvf\bigr)}^d}.\label{eq:expr-volume-from-radius-vector-function}
\intertext{In particular, in the case $\log{\la\ka_d}=\frac d2\log{\frac d{2x}}+o(d)$ with $x>0$, which corresponds to the critical regime considered by Bonnet, Kabluchko and Turchi~\cite{Bonnet21},
we can recover the asymptotics for the volume ratio stated in their Theorem~3.1. Indeed,
for any subsequence of~$d\to\infty$
we have, almost surely along a further sub-subsequence,}
{\bigl(\rvf\bigr)}^d&=\exp\left[d\log\left(1-\frac12{(\la\ka_d)}^{-\frac2{d+1}}\bigl(1+o(1)\bigr)\right)\right]\notag\\[.4em]
&=\exp\left[d\log\left(1-\frac xd\,\bigl(1+o(1)\bigr)\right)\right]\notag\\[.4em]
&\sim\ee^{-x},\notag
\end{align}
using \cref{cor:radius-vector-function}, part (ii). This together with the dominated convergence theorem yields
\[\lim_{d\to\infty}\frac{\EE{\lvert{K^d_\la\rvert}}}{\ka_d}=\ee^{-x}.\]
In view of~\eqref{eq:expr-volume-from-radius-vector-function} and \cref{cor:radius-vector-function}, part~(ii), we can conjecture that when $d\ll \log(\la \ka_d)\ll \frac{d}2\log(d)$, 
\begin{equation*}
\log{\frac{\EE{\lvert{K^d_\la\rvert}}}{\ka_d}}\sim -\frac{d}2{(\la \ka_d)}^{-\frac2{d+1}}.   
\end{equation*}
Similarly, when $\log(\la\ka_d)\gg \frac{d}2\log(d)$, we also predict that
\begin{equation*}
1-\frac{\EE{\lvert{K^d_\la\rvert}}}{\ka_d}\sim\frac{d}2(\la\ka_d)^{-\frac2{d+1}}.
\end{equation*}
We thank Matthias Reitzner for his suggestion which led us to this observation.
\end{remark}

\appendix
\section*{Appendix}\label{sec:appendix}
In this appendix, we state and prove \cref{pro:extension-janson}, which is instrumental in deriving \cref{thm:main}.
As exposed in \cref{sec:janson-covering}, Janson's result~\cite[Lemma~8.1]{Janson86}
says that, under~\eqref{eq:moment-janson}
and~\eqref{eq:condition-janson},
the probability $\PP\bigl(\Cover(\La(a),R,a;M)\bigr)$
of covering the manifold~$M$ by a union of geodesic balls,
$\bigcup_iB_M(x_i,a\rho_i)=M$,
where the~$(x_i,\rho_i)$'s arise as the atoms of a Poisson point process with intensity
$\La(a)v_M(\dd x)\otimes\PP(R\in\dd\rho)$,
converges to the Gumbel distribution function as~$a\to0$.
Our aim is to prove the following extension
where the random radius~$R$ is allowed to depend on~$a$.
Instead of adapting Janson's result and rewriting the whole proof,
we rather reduce to it using a coupling argument.
Namely, if~$R_a$ converges to~$R$ in such a way that we may construct~$(R,R_a)$
so that
\begin{equation}
  (1-\eta)R\le R_a\le(1+\eta)R\label{eq:coupling-property}
\end{equation}
holds with high probability for some positive sequence~$\eta:=\eta(a)$
vanishing sufficiently fast, then we may approximate (from above and below)
$\PP\bigl(\Cover(\La,R_a,a;M)\bigr)$
with similar probabilities where~$R_a$ is changed to~$(1\pm\eta)R$,
which after replacing~$a$ by~$a/(1\pm\eta)$
leads to the case handled by Janson.

In \cref{pro:extension-janson} below, condition~\eqref{eq:uniform-radius-upper-bound}
allows us to achieve the upper bound in~\eqref{eq:coupling-property} and
is of course satisfied if the stronger condition
$(\forall\rho>0,\ \PP(R_a>\rho)\le\PP(R>\rho))$ holds for~$a>0$ sufficiently small
(e.g., if the sequence of functions $\rho\mapsto\PP(R_a>\rho)$ is
eventually non-decreasing as~$a\to0$),
while condition~\eqref{eq:W1-radius-convergence} allows us to obtain the lower
bound in~\eqref{eq:coupling-property}.

\setcounter{proposition}{0}
\renewcommand{\theproposition}{\Alph{proposition}}
\begin{proposition}[Extension of Janson's result]\label{pro:extension-janson}
  Let~$M$ be a~$\cC^2$, $D$-dimensional compact Riemannian manifold, and let $\La:=\La(a)>0$ and~$R$
  fulfill Janson's conditions~\eqref{eq:moment-janson} and~\eqref{eq:condition-janson}.
  Suppose $R_a,\,a>0,$ are positive random variables such that
  for every~$w>0$ and all~$a$ sufficiently small,
  \begin{equation}
    \sup_{\rho>0}\:\frac{\PP(R_a>\rho)}{\PP\left(\left[1+\frac w{\log{\frac1a}}\right]R>\rho\right)}\le1.
    \label{eq:uniform-radius-upper-bound}
  \end{equation}
  Suppose further that, as~$a\to0$,
  \begin{equation}
    \Wd\left(\log{R_a},\log R\right)=o\left(\frac1{\log^2{\frac1a}}\right).
    \label{eq:W1-radius-convergence}
  \end{equation}
  Then~\eqref{eq:janson-gumbel}
  also holds with~$R_a$ in place of~$R$:
\[\lim_{a\to0}\PP\bigl(\Cover(\La,R_a,a;M)\bigr)
    =\ee^{-\ee^{-\tau}}\!.\]
\end{proposition}
\begin{proof}
  First, with the assumptions~\eqref{eq:moment-janson} and~\eqref{eq:condition-janson}
  of Janson's theorem~\cite[Lemma~8.1]{Janson86} fulfilled for~$R$,
  the convergence~\eqref{eq:janson-gumbel} holds:
  \[\lim_{a\to0}\PP\bigl(\Cover(\La,R,a;M)\bigr)=\ee^{-\ee^{-\tau}}.\]
  To prove the same for $\Cover(\La,R_a,a;M)$, we use a coupling argument.
  Let~$w>0$ and \begin{equation}
    \eta_a:=\frac w{\log{\frac1a}}.\label{eq:looseness_parameter}
  \end{equation}
  By~\eqref{eq:uniform-radius-upper-bound},
  it holds for all~$a$ sufficiently small that
  \[\forall\rho>0,\qquad\PP(R_a>\rho)\le\PP\bigl((1+\eta_a)R>\rho\bigr).\]
  Applying the inverse method,
  that is, considering the generalized inverses~$F_a^{-1}$ and~$F^{-1}$ of the
  distribution functions of~$R_a$ and~$R$ respectively,
  and setting $R_a:=F_a^{-1}(U)$ and~$R:=F^{-1}(U)$ for a uniform variable~$U$ in~$\cc{0}{1}$,
  we may then suppose for every small~$a>0$  that~$R_a$ and~$R$ are coupled so that
  \[R_a\le(1+\eta_a)R,\quad\text{almost surely}.\]
  We then introduce independent, uniformly distributed variables~$X_i,\,i\ge1,$ on~$M$
  and, for every small~$a>0$, an independent Poisson-distributed variable~$N_a$
  with mean $\La(a)v_M(M)$,
  as well as a further independent family $(R_{a,i},R_i)_{i\ge1}$ of i.i.d.\ copies of~$(R_a,R)$.
  Thus, for every small~$a>0$, we have constructed a Poisson point process
  $\bs\Xi_a:=\{X_i,R_{a,i},R_i\}_{1\le i\le N_a}$
  on~$M\times\oo{0}{\infty}^2$ whose projections
  $\Xi_{1,a}:=(X_i,R_{a,i})_{1\le i\le N_a}$ and $\Xi_{2,a}:=(X_i,R_i)_{1\le i\le N_a}$ have
  intensity $\La v_M(\dd x)\otimes\PP(R_a\in\dd\rho)$ and $\La v_M(\dd x)\otimes\PP(R\in\dd\rho)$
  respectively. In particular, with this construction the two covering events are given by
  \[\Cover(\La,R_a,a;M)=\left\{\bigcup_{\Xi_{1,a}}B_M(X_i,aR_{a,i})=M\right\}
  \enspace\text{and}\enspace
    \Cover(\La,R,a;M)=\left\{\bigcup_{\Xi_{2,a}}B_M(X_i,aR_i)=M\right\}.\]
  Since
  $R_{a,i}\le(1+\eta_a)R_i$ almost surely for all~$i\ge1$,
  we get
  \begin{align*}
    \PP\bigl(\Cover(\La,R_a,a;M)\bigr)
    &=\PP\left(\bigcup_{\Xi_{1,a}}B_M(X_i,aR_{a,i})=M\right)\\[.4em]
    &\le\PP\left(\bigcup_{\Xi_{2,a}}B_M(X_i,a(1+\eta_a)R_i)=M\right)\\[.4em]
    &=\PP\bigl(\Cover(\La,R,(1+\eta_a)a;M)\bigr),
  \end{align*}
  and therefore,
  \begin{align*}
  \limsup_{a\to0}\:\PP\bigl(\Cover(\La,R_a,a;M)\bigr)
  &\le\limsup_{a\to0}\:\PP\bigl(\Cover(\La,R,(1+\eta_a)a;M)\bigr)\\[.4em]
  &=\limsup_{a\to0}\:\PP\bigl(\Cover(\La^-,R,a;M)\bigr),
  \end{align*}
  with $
    \La^-(a):=\La(\frac a{1+\eta_a})$ and
  the last two~$\limsup$ being equal because $a \mapsto a/(1+\eta_a)$ can be inverted when~$a>0$ is sufficiently small.
  This~$\limsup$ is in fact a true limit, as we now show.
  Recalling $J(\La,R,a;M)\to\tau$ by~\eqref{eq:condition-janson} and $\eta_a:=w/{\log{\frac1a}}$ in~\eqref{eq:looseness_parameter},
  \begin{align*}
    \La^-(a)
    &=\frac{{(1+\eta_a)}^D}{b(R;M)a^Dv_M(M)}
    \left(\log{\frac{{(1+\eta_a)}^D}{b(R;M)a^D}}
    +D\log{\log{\frac{{(1+\eta_a)}^D}{b(R;M)a^D}}+\log{\alpha(R)}}+\tau+o(1)\right)\\[.4em]
    &=\frac{1+D\eta_a+o(\eta_a)}{b(R;M)a^Dv_M(M)}
    \left(\log{\frac1{b(R;M)a^D}}
    +D\log{\log{\frac1{b(R;M)a^D}}+\log{\alpha(R)}}+\tau+o(1)\right)\\[.4em]
    &=\left[1+\frac{Dw}{\log{\frac1a}}+o\left(\frac1{\log{\frac1a}}\right)\right]\La(a).
  \end{align*}
  This entails $J(\La^-,R,a;M)\to\tau+D^2w$, so Janson's theorem applies with
  \[\lim_{a\to0}\:\PP\bigl(\Cover(\La^-,R,a;M)\bigr)=\ee^{-\ee^{-\tau-D^2w}}\!.\]
  Letting~$w\to0^+$,
  we have proved
  \[\limsup_{a\to0}\:\PP\bigl(\Cover(\La,R_a,a;M)\bigr)\le\ee^{-\ee^{-\tau}}\!.\]
  To establish the other direction, we keep~$w>0$ and $\eta_a:=w/{\log{\frac1a}}$,
  as well as the Poisson point process
  $\bs\Xi_a:=\{X_i,R_{a,i},R_i\}_{1\le i\le N_a}$
  and its projections~$\Xi_{1,a}$ and~$\Xi_{2,a}$,
  except that this time $(R_{a,i},R_i)_{i\ge1}$ are i.i.d.\ copies
  of a different coupling of~$(R_a,R)$.
  Namely, the Kantorovich--Rubinstein theorem~\cite[Theorem~11.8.2]{Dudley02}
  allows us to choose~$(R_a,R)$ so that
  \begin{equation}
    \Wd\left(\log{R_a},\log R\right)=\EE{\left|\log{R_a}-\log R\right|}.%
    \label{eq:kantorovich-rubinstein}
  \end{equation}
  We then restrict~$\bs\Xi_a$
  by keeping only the radii~$R_{a,i}$ such that
  $R_{a,i}\ge(1-\eta_a)R_i$:
  \[\widetilde{\bs\Xi}_a:=\Bigl\{(X_i,R_{a,i},R_i)\in\bs\Xi_a:
  R_{a,i}\ge(1-\eta_a)R_i\Bigr\}.\]
  Hence~$\widetilde\Xi_{2,a}:=\{(X_i,R_i):(X_i,R_{a,i},R_i)\in\widetilde{\bs\Xi}_a\}\subseteq\Xi_{2,a}$
  has a smaller intensity,
  $\widetilde\La v_M(\dd x)\otimes\PP(R_a\in\dd\rho)$
  with $\widetilde\La(a):=\La(a)\PP(R_a\ge(1-\eta_a)R)$.
  Since, by Markov's inequality,~\eqref{eq:kantorovich-rubinstein} and~\eqref{eq:W1-radius-convergence},
  \begin{align*}
  \PP(R_a<(1-\eta_a)R)
  &=\PP\left(\log{\frac R{R_a}}>\log{\frac1{1-\eta_a}}\right)\\[.4em]
  &\le\frac1{\log{\frac1{1-\eta_a}}}\EE{\left|\log{R_a}-\log R\right|}\\[.4em]
  &\le\frac{\log{\frac1a}}w\Wd\left(\log{R_a},\log R\right)\\[.4em]
  &=o\left(\frac1{\log{\frac1a}}\right),
  \end{align*}
  we have $\widetilde\La(a)=\La(a)[1+o(\log^{-1}(\frac1a))]$,
  which by~\eqref{eq:condition-janson} implies that
  $J\bigl(\widetilde\La,R,a;D\bigr)\to\tau$, i.e.,
  \begin{equation}
    \widetilde\La(a)=\frac1{b(R;M)a^Dv_M(M)}\left(\log{\frac1{b(R;M)a^D}}
    +D\log{\log{\frac1{b(R;M)a^D}}+\log{\alpha(R)}}+\tau+o(1)\right).\tag{$\widetilde J_2$}%
    \label{eq:condition-janson-tilde}
  \end{equation}
  Now,
  \begin{align*}
    \PP\bigl(\Cover(\La,R_a,a;M)\bigr)
    &=\PP\left(\bigcup_{\Xi_{1,a}}B_M(X_i,aR_{a,i})=M\right)\\[.4em]
    &\ge\PP\left(\bigcup_{\widetilde{\bs\Xi}_a}B_M(X_i,aR_{a,i})=M\right)\\[.4em]
    &\ge\PP\left(\bigcup_{\widetilde\Xi_{2,a}}B_M(X_i,a(1-\eta_a)R_i)=M\right)\\[.4em]
    &=\PP\bigl(\Cover(\La^+,R_i,(1-\eta_a)a;M)\bigr)
  \end{align*}
  and therefore (in the same way as above),
  \[\liminf_{a\to0}\:\PP\bigl(\Cover(\La,R_a,a;M)\bigr)
  \ge\liminf_{a\to0}\:\PP\bigl(\Cover(\La^+,R,a;M)\bigr),\]
  where, recalling~\eqref{eq:condition-janson-tilde} and~$\eta_a:=w/{\log{\frac1a}}$,
  \begin{align*}
    \La^+(a)
    &:=\widetilde\La\left(\frac a{1-\eta_a}\right)\\[.4em]
    &=\frac{{(1-\eta_a)}^D}{b(R;M)a^Dv_M(M)}
    \left(\log{\frac{{(1-\eta_a)}^D}{b(R;M)a^D}}
    +D\log{\log{\frac{{(1-\eta_a)}^D}{b(R;M)a^D}}+\log{\alpha(R)}}+\tau+o(1)\right)\\[.4em]
    &=\frac{1-D\eta_a+o(\eta_a)}{b(R;M)a^Dv_M(M)}
    \left(\log{\frac1{b(R;M)a^D}}
    +D\log{\log{\frac1{b(R;M)a^D}}+\log{\alpha(R)}}+\tau+o(1)\right)\\[.4em]
    &=\left[1-\frac{Dw}{\log{\frac1a}}+o\left(\frac1{\log{\frac1a}}\right)\right]\La(a).
  \end{align*}
  This entails $J(\La^+,R,a;M)\to\tau-D^2w$, so Janson's theorem applies again with
  \[\lim_{a\to0}\:\PP\bigl(\Cover(\La^+,R,a;M)\bigr)=\ee^{-\ee^{-\tau+D^2w}}\!.\]
  It remains to let~$w\to0^+$ to complete the proof.
\end{proof}

\begingroup
\providecommand{\bysame}{\leavevmode\hbox to3em{\hrulefill}\thinspace}
\renewcommand{\MR}[1]{%
  \href{https://www.ams.org/mathscinet-getitem?mr=#1}{MR\,#1}
}
\newcommand{\arXiv}[1]{%
  \href{https://arxiv.org/abs/#1}{arXiv:#1}
}

\endgroup

\end{document}